\documentclass[11pt]{article}
\usepackage{authblk}
\usepackage{amssymb,latexsym,amsmath,graphicx,amsfonts,amsthm,xy,eucal,mathrsfs,fullpage, epsfig}
\usepackage{amscd,subfigure,color,xcolor,setspace,cite,mathtools}
\usepackage[font=footnotesize]{caption}
\usepackage{float}
\usepackage{pgfplots}
\usepackage{epstopdf}
\usepackage{tikz-cd}
\usepackage[english]{babel}
\usepackage{babel}
\usepackage{fontenc}
\usepackage[colorlinks=true,citecolor=blue]{hyperref}
\usepackage{float}
\usepackage{lineno}
\usepackage{pgf,tikz}
\usepackage{array}
\usepackage{multirow}
\usepackage{float}
\restylefloat{table}

\usetikzlibrary{arrows}
\usetikzlibrary[patterns]
\usetikzlibrary{decorations.pathmorphing}
 \usetikzlibrary{plotmarks}
        \usetikzlibrary{intersections}
        \usetikzlibrary{shadings}
        \usetikzlibrary{calc}
\newlength{\hatchspread}
\newlength{\hatchthickness}
\newlength{\hatchshift}
\newcommand{\hatchcolor}{}
\tikzset{hatchspread/.code={\setlength{\hatchspread}{#1}},
         hatchthickness/.code={\setlength{\hatchthickness}{#1}},
         hatchshift/.code={\setlength{\hatchshift}{#1}},
         hatchcolor/.code={\renewcommand{\hatchcolor}{#1}}}
\tikzset{hatchspread=3pt,
         hatchthickness=0.4pt,
         hatchshift=1pt,
         hatchcolor=black}
\pgfdeclarepatternformonly[\hatchspread,\hatchthickness,\hatchshift,\hatchcolor]
   {custom north west lines}
   {\pgfqpoint{\dimexpr-2\hatchthickness}{\dimexpr-2\hatchthickness}}
   {\pgfqpoint{\dimexpr\hatchspread+2\hatchthickness}{\dimexpr\hatchspread+2\hatchthickness}}
   {\pgfqpoint{\dimexpr\hatchspread}{\dimexpr\hatchspread}}
   {
    \pgfsetlinewidth{\hatchthickness}
    \pgfpathmoveto{\pgfqpoint{0pt}{\dimexpr\hatchspread+\hatchshift}}
    \pgfpathlineto{\pgfqpoint{\dimexpr\hatchspread+0.15pt+\hatchshift}{-0.15pt}}
    \ifdim \hatchshift > 0pt
      \pgfpathmoveto{\pgfqpoint{0pt}{\hatchshift}}
      \pgfpathlineto{\pgfqpoint{\dimexpr0.15pt+\hatchshift}{-0.15pt}}
    \fi
    \pgfsetstrokecolor{\hatchcolor}
    \pgfusepath{stroke}
   }
\usetikzlibrary{arrows}
\usetikzlibrary[patterns]
\numberwithin{equation}{section}
\newtheorem{theorem}{Theorem}[section]
\newtheorem{definition}[theorem]{Definition}
\newtheorem{lemma}[theorem]{Lemma}
\newtheorem{example}[theorem]{Example}
\newtheorem{proposition}[theorem]{Proposition}
\newtheorem{corollary}[theorem]{Corollary}
\newtheorem{remark}{Remark}[section]

\def\bbr{{\mathbb R}}

\begin{document}\title{\textbf{Attractors in  $k$-dimensional discrete systems of mixed monotonicity}}
\author[1, 2]{ Ziyad AlSharawi\thanks{Corresponding author: zsharawi@aus.edu. This work was done while the first author was on sabbatical leave from the American University of Sharjah.}}
\author[1]{ Jose S. C\'anovas}
\author[2]{ Sadok Kallel}

\affil[1]{\small Universidad Politécnica de Cartagena, Paseode Alfonso XIII 30203, Cartagena, Murcia, Spain}
\affil[2]{\small American University of Sharjah, P. O. Box 26666, University City, Sharjah, UAE}
\maketitle

\begin{abstract}
We consider $k$-dimensional discrete-time systems of the form  $x_{n+1}=F(x_n,\ldots,x_{n-k+1})$ in which the map $F$ is continuous and monotonic in each one of its arguments. We define a partial order on $\mathbb{R}^{2k}_+$, compatible with the monotonicity of $F$, and then use it to embed the $k$-dimensional system into a $2k$-dimensional system that is monotonic with respect to this poset structure. An analogous construction is given for periodic systems.  Using the characteristics of the higher-dimensional monotonic system, global stability results are obtained for the original system. Our results apply to a large class of difference equations that are pertinent in a variety of contexts. As an application of the developed theory, we provide two examples that cover a wide class of difference equations, and in a subsequent paper, we provide additional applications of general interest.
\end{abstract}


\section{ Introduction}
We begin by providing a motivational and fundamental example of a one-dimensional map that illustrates the notion we aim to generalize in this paper. Let $f$ be a continuous decreasing function that maps an interval $I\subset\bbr$ into itself. We write $f(\downarrow)$, with the arrow indicating the monotonicity. When $f$ has no cycles of period two, a sequence of iterates of $f$ (i.e., solutions of $x_{n+1}=f(x_n)$) spirals in and converges to a fixed point of $f$ as shown in Fig. \ref{Fig-illustration1}.
This figure depicts a fundamental principle, which encompasses the core concepts that we aim to generalize and transform into practical mechanisms.
\begin{figure}[htpb]
\begin{center}
\begin{tikzpicture}[line cap=round,line join=round,>=triangle 45,x=1.0cm,y=1.0cm,scale=0.6]
\draw[help lines,step=.5] (0,0) grid (10,10);
\draw[-triangle 45, line width=1.0pt,scale=1] (0,0) -- (10.0,0) node[below] {$x_n$};
\draw[line width=1.0pt,-triangle 45] (0,0) -- (-1.0,0);
\draw[-triangle 45, line width=1.0pt,scale=1] (0,0) -- (0,10.0) node[left] {$x_{n+1}$};
\draw[line width=1.0pt,-triangle 45] (0,0) -- (0.0,-1.0);
\draw[line width=0.8pt,color=green] (0,0) -- (10.0,10.0);
\draw[line width=1.2pt,domain=0:10.0,smooth,variable=\x,color=blue] plot ({\x},{9-3*ln(1+\x)});
\draw[line width=1.0pt,color=red,smooth] (0.5, 7.783604676)--(7.78,7.78)--(7.783604676,2.481339363)--(2.48,2.48)--(2.481339363, 5.257748718)--(5.26,5.26)--(5.257748718, 3.498538527)--(3.498,3.498)--(3.498538527, 4.488742284)--(4.49,4.49)--
(4.488742284, 3.891902589)--(3.9,3.9)--(3.891902589, 4.237256085)--(4.23,4.23)--(4.237256085, 4.032606861);
\end{tikzpicture}
\end{center}
    \caption{The ``Cobweb convergence'' of a sequence $x_{n+1}=f(x_n)$, with $f$ decreasing. The sequence is embedded along the diagonal in the form $(x_n,x_n)$. When $f$ has no cycles of length two, the converging spiral illustrates ``box convergence'', which we develop for higher-dimensional maps (see \S\ref{embeddingtheory}).}\label{Fig-illustration1}
\end{figure}
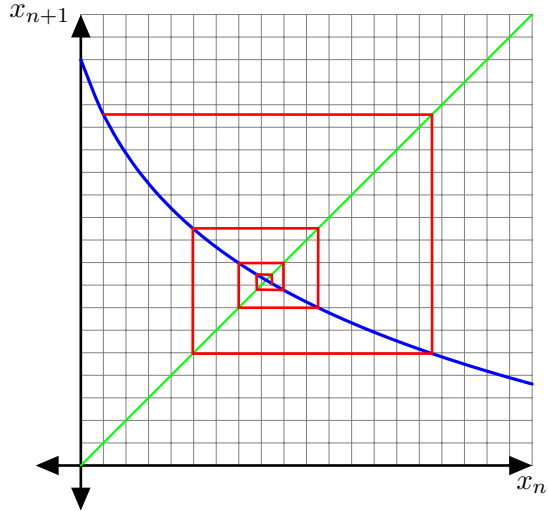

Define on $\bbr_+^2$ the so-called south-east partial ordering ($\leq_{\lambda}$)
\begin{equation}\label{southeast}
(x_1,x_2)\leq_{\lambda} (y_1,y_2)\ \Longleftrightarrow\ x_1\leq y_1\quad \text{and}\quad  y_2\leq x_2.
\end{equation}
Then define the two dimensional map $G: \bbr_+^2\to \bbr_+^2$  by
\begin{equation}\label{Eq-Gr1}
   G(x,y) = (f(y), f(x)).
\end{equation}

Because the function $f$ is decreasing, the function $G$ is non-decreasing with respect to $\leq_{\lambda},$ i.e. for $X,Y\in \bbr_+^2$,
$ X\leq_\lambda Y\ \Longrightarrow\ G(X)\leq_\lambda G(Y)$.
In other words, if $X=(x_1,y_1)$, $Y=(x_2,y_2)$, with $x_1\leq x_2$ and $y_1\geq y_2$, then $f(x_1)\geq f(x_2)$, and $f(y_1)\leq f(y_2)$.
Now, choose $a<b$, $P=(a,b)$ and $Q=P^t=(b,a)$ such that
$$P<_{\lambda}G(P)<_\lambda G(Q)<_\lambda Q$$
(note that $G(P)<_\lambda G(Q)$ is automatic since $P<_\tau Q$).
We call this \textit{a trapping box} (see \S \ref{boxconvergence}).
Pick $X_0=(x_0,x_0)$ in $[a,b]^2$ so that
$P\leq_{\lambda} X_0\leq_{\lambda} Q$.
 Let $x_0$ be an initial condition of $x_{n+1}=f(x_n)$, and iterate the map $G$ on the points $P,Q$ and $X_0=(x_0,x_0)$ to obtain
$$P\leq G^n(P)\leq G^n(x_0,x_0) = (x_n,x_n)\leq G^n(Q)\leq Q\ \ \ ,\ \ \ n\geq 1.$$
The nested red spiral path illustrates this in Fig. \ref{Fig-illustration1}. If $G$ has a unique fixed point and, both sequences $\{G^n(P)\}_{n\geq 0}$ and $\{G^n(Q)\}_{n\geq 0}$ are converging to the same point, it follows that the sequence $(x_n,x_n)$, and thus $x_n$, is converging as well.
\\

The notion of transitioning to a higher-dimensional system and employing partial ordering to ensure convergence, as seen above, is not new; for instance, see \cite{Go-Ha1994,Sm2006,Sm2008,Al2022} and the references therein. However, it is neither sufficiently developed nor optimally exploited. So we aim to remedy this deficiency and develop a general theory that applies to any difference equations of the form
\begin{equation}\label{Eq-GeneralCase1}
x_{n+1}=F(x_n,x_{n-1},\ldots,x_{n-k+1}),\; k\geq 1,\quad \text{where}\quad F: V\to \mathbb{R}_+.
\end{equation}
Also, $F$ is monotonic in each one of its components (increasing or decreasing) and $V$ can be any closed subset of $\bbr$ (in our case, we work with $V=\bbr, \mathbb{R}_+$ or $[a,b]$). Our approach hinges on associating to each monotonicity pattern $\uparrow_\tau$ of $F$ a partial ordering $\leq_\tau$ on $V^k$ as in \eqref{southeast}, and a suitable ``diagonal'' extension of \eqref{Eq-GeneralCase1} to $V^k\times V^k$, as in \eqref{Eq-Gr1}. This extended system is monotonic with respect to $\leq_\tau$, therefore enjoying convergence properties under mild conditions (existence of trapping boxes and uniqueness of fixed points). As usual, $X\leq_\tau Y$  means $X<_\tau Y$ or $X=Y.$ These ideas are generalizable to $p$-periodic difference equations
\begin{equation}\label{Eq-GeneralCase1Periodic}
x_{n+1}=F_n(x_n,x_{n-1},\ldots,x_{n-k+1}),\; k\geq 1,\quad \text{where}\quad F_n:\mathbb{R}_+^k\to \mathbb{R}_+.
\end{equation}
Note that the case $k=2$ is well-treated in the literature \cite{Sm2006,Sm2008,Al2022,Al-Ka2023}, and it was the motivation for this paper.
\\

The origin of the embedding approach and the idea of decomposing an operator into two monotonic components goes back to numerical analysis. This can be traced back to a paper by Schr\"{o}der \cite{Sc1959} and a book by Collatz \cite{Co1964}, see \cite{Go-Ha1994,Sm2006,Sm2008} and the references therein for a good account of history citation. The approach has been effectively utilized across different research areas \cite{Al2022,Al-Ka2023,El-He2021,Ku-Or2006}. The novelty in the embedding notion resides in defining a suitable partial order and an appropriate extension. A known partial order is to consider a convex closed cone $V,$ and define $X<_\lambda Y$ whenever $Y-X\in V.$ For instance, this is found to be effective when the map $F$ of Eq. \eqref{Eq-GeneralCase1} is non-decreasing in each one of its arguments \cite{Kr-Pi2004}. However, it fails to address the issue of mixed monotonicity.    In the subsequent sections, we extensively elaborate on the aforementioned concepts and provide results in the most inclusive manner feasible.
\\

The structure of this paper is as follows: In Section 2, we define a general partial order compatible with the monotonicity of the $k$-dimensional map $F,$ then embed the dynamical system into a higher dimensional dynamical system by constructing an extension of $F$ to a $2k$-dimensional map that is non-decreasing under the new partial order (Proposition \ref{keyembedding}). We give results on global stability that extend the scope of some previously established results for dimension two (Theorem \ref{nested} and Theorem \ref{Th-GlobalStability}). In Section 3, we extend the results of section two to periodic monotonic systems (Theorem \ref{Th-GlobalStability2}). In Section 4, we demonstrate the significance of our developed theory through some applications that span a wide class of difference equations.  In the last section, we give a conclusion that summarizes the main findings of this paper.


\section{Embedding discrete-time systems}\label{embeddingtheory}

Let $(V,\leq)$ be a partially ordered metric space, in which $V$ can be  $\bbr_+$ or $[a,b].$
A recursive sequence in $V$, with delay $k$, is any sequence defined by
\begin{equation}\label{recursive}
\alpha_F : x_{n+1} = F(x_{n},\ldots, x_{n-k+1}), \ \ k\geq 1,
\end{equation}
where $F:V^{k}\rightarrow V$ is a continuous function and $X_0:=(x_0,x_{-1},\ldots, x_{-k+1})$ is a given initial condition in $V$.
We write $\mathcal S(V)$ to denote the set of all recursive sequences in $V$.  Different functions $F$ can give rise to the same recursive sequence \eqref{recursive}, so the pair $(F,X_0)$ does not uniquely determine the system in general; only the sequence does.
\\

An injection $\Psi: \mathcal S(V_1)\hookrightarrow\mathcal S (V_2)$ that sends convergent
sequences to convergent sequences is called an embedding.  In general, for an embedding $\Psi: \mathcal S(V_1)\rightarrow \mathcal S(V_2)$, the convergence of $\Psi (\alpha)$ does not imply the convergence of $\alpha$. When it does, we call it \textit{a strong} embedding.
Our approach now consists in finding a suitable embedding $\Psi : \mathcal S(V)\hookrightarrow\mathcal S(V^{n})$, for some $n$, so that the convergence of $\Psi (\alpha)$ can shed light on the convergence of $\alpha$.
We shall describe a specific construction to accomplish this objective by employing diagonal embeddings and partial orderings. We vastly expand on known ideas and state certain results as generally as possible.
\\

Rewrite the system \eqref{recursive}  in vector form as $\beta_T : X_{n+1} = T(X_n)$ and \begin{equation}\label{Eq-mapT}
T(X_n)= (F(X_n),x_{n},\ldots, x_{n-k+2})\ ,\ \hbox{where}\ X_n = (x_n,\ldots, x_{n-k+1})
\end{equation}
with an initial condition $X_0 := (x_0,x_{-1},\ldots, x_{-k+1})$.
Associating to $\alpha_F$ the sequence $\beta_T$ gives a strong embedding
$\mathcal S(V)\hookrightarrow\mathcal S(V^k)$, meaning again that $\beta_T$ converges if and only if $\alpha_F$ does.

\begin{definition}
    Let $T: V^k\rightarrow V^k$ be as in Eq. \eqref{Eq-mapT} and $r\geq 2$. A ``diagonal embedding'' of $T$ is a map $G: (V^k)^r\rightarrow (V^k)^r$ such that the restriction of $G$ to the diagonal is ${(T,\ldots, T)}$ ($r$-times).
\end{definition}

In this paper, we confine ourselves to $r=2$ and $G: V^k\times V^k\rightarrow V^k\times V^k$. Now, we build partial orderings on this product starting from a partial order $\leq$ on $V$. We write $(V,\leq)$ when we want to stress that $V$ is equipped with $\leq$, otherwise, we write $V$ ($\leq$ being implicitly understood).

\begin{definition}\label{Def-Poset}\rm (Monotonic orderings). Let $(V,\leq)$ be a partially ordered set (a poset). For any map $\tau : \{1,2,\ldots,k\}\rightarrow\{-1,1\}$, we define a partial ordering $\leq_\tau$ on $V^k$ by
$$(x_1,\ldots, x_k)\leq_\tau (y_1,\ldots, y_k)\
\Longleftrightarrow\ \forall i,\
\begin{cases} x_i\leq y_i\ ,&\hbox{if}\ \tau (i)=1\\
x_i\geq y_i\ ,&\hbox{if}\ \tau (i)=-1.
\end{cases}$$
If $\tau$ is a map as above, we write $\tau^\vee$ its dual defined by $\tau^\vee = -\tau$. The ``duality'' simply reflects the following fact: if $X,Y\in V^k$, then
$X\leq_\tau Y$ if and only if $Y\leq_{\tau^\vee}X$. We will write interchangeably $(V^k,\leq_\tau)$ or $(V^k,\tau)$ to denote the product $V^k$ endowed with the ordering $\leq_\tau$. Writing $X<_\tau Y$ means $X\leq_\tau Y,$ but $X\neq Y.$
\end{definition}

The simplest examples of a monotonic ordering are the
south-east ordering on $V^2$,  given by \eqref{southeast}, and its dual the north-west ordering. The south-east ordering corresponds to
$\tau : \{1,2\}\rightarrow \{-1,1\}$,
$\tau  (1)=1$ and $\tau (2)=-1$. As we explain shortly, this ordering is associated with maps of the form $F(\uparrow,\downarrow).$
Similarly, when $k=3$, the map $\tau (1)=1,\tau (2)=-1$ and $\tau (3) = 1$ corresponds to the ordering
$$(x,y,z)\leq_{\tau}(u,v,w)\ \Longleftrightarrow\ x\leq u\ ,\ y\geq v\ ,\ z\leq w.$$
This ordering is associated with maps of the form $F(\uparrow,\downarrow,\uparrow).$
\\

With the partial order $\tau$ on $V^k$, we can turn $V^k\times V^k$ into a poset with the south-east ordering $\lambda$ defined as in \eqref{southeast}, i.e.,
$$(X_1,U_1)\leq_{\lambda}(X_2,U_2)\ \Longleftrightarrow\ X_1\leq_{\tau}X_2\ \hbox{and}\ U_2\leq_{\tau} U_1.$$
This poset structure on $V^k\times V^k=V^{2k}$ is given by the poset product with ordering $\tau\times\tau^\vee$. For convenience, we may write $V^k\times_{\lambda}V^k$ to indicate the pair ($V^k\times V^k, \tau\times \tau^\vee)$. Based on our southeast order $\lambda,$ we can write $\lambda=\tau\times\tau^\vee$ in general.

\subsection{Trapping box}\label{boxconvergence}

The \textit{monotone convergence theorem} applies to $(V^k, \tau)$ when $V$ is a closed subspace of $\bbr$.
This implies some useful squeeze-type results that are fundamental when applying our technique.

\begin{lemma}\label{monotone} (Monotone Convergence)
Let $V$ be a closed subset of $\bbr$ (eg. $V=\bbr,\bbr^+$ or $V=[a,b]$), and consider on $V^k$ any monotone ordering $\tau$.  Suppose $G$ is increasing (resp. decreasing) with respect to the partial order $\tau$ and is bounded above (resp. below) in $\tau$. Define the sequence $\zeta_{G} : x_{n+1} = G(x_n)$, $n\geq 0,$ then $\zeta_G$ converges to a fixed point of $G$.
\end{lemma}

\begin{proof} By projecting the sequence onto the $i$-th factor, we obtain a sequence $\zeta^i_G$ in $V$ that is monotonic and bounded. By the monotone sequence theorem valid for $V\subseteq\bbr$, this $i$-th sequence must converge to $y_i\in V$. It follows that $\zeta_G$ converges to $(y_1,\ldots, y_k)$, and this must be a fixed point of $G.$
\end{proof}

\begin{lemma}\label{squeeze} (Squeeze Lemma)
  Let $(V^k,\tau )$ as in Lemma \ref{monotone} and $\zeta_{G} : x_{n+1} = G(x_n)$, $n\geq 0$, with $x_0$ an initial point in $V^k$. Assume we have a ``trapping box'', that is there are two points $p, q$ in $V^k$ such that $p\leq_\tau x_0\leq_\tau q$, and $p\leq_\tau G(p)$, $G(q)\leq_\tau q$. \\
  (a) Suppose $G$ is increasing with respect to the partial order $\leq_\tau$. If $G$ has a unique fixed point, then the sequence $\zeta_{G}$ converges to that unique fixed point. \\
  (b) Suppose $G$ is decreasing with respect to $\leq_\tau$. If $G$ has no $2$-cycles, then the sequence $\zeta_{G}$ either converges to a fixed point of $G,$ or it is eventually bounded between two fixed points of $G.$
\end{lemma}

\begin{proof}
(a) If $G$ is increasing,  $p\leq_\tau G^n(p)\leq_\tau G^n(x_0)\leq_\tau G^n(q)\leq_\tau q$, $\forall n\geq 1$.
The sequence $\{G^n(p)\}_{n\geq 0}$ is increasing, and the sequence $\{G^n(q)\}_{n\geq 0}$ is decreasing, so both converge to the unique fixed point, which must also be the limit of $\zeta_{G} = \{G^n(x_0)\}_{n\geq 0}$. (b)
By induction, we obtain two bounded decreasing sequences, namely $\{G^{2n}(q)\}$ and $\{G^{2n+1}(p)\},$ and we obtain two bounded increasing sequences, namely $\{G^{2n}(p)\}$ and $\{G^{2n+1}(q)\}.$
Furthermore, the sequence $\{G^{n}(x_0)\}$ is squeezed between $\{G^n(p)\}$ and $\{G^n(q)\}.$ In particular,
$$G^{2n}(p)\leq_\tau G^{2n}(x_0)\leq_\tau G^{2n}(q)\quad \text{and}\quad G^{2n+1}(q)\leq_\tau G^{2n+1}(x_0)\leq_\tau G^{2n+1}(p). $$
Therefore, we must have  $\{G^n(p)\}\to \{\bar p_1,\bar p_2\}$ and $\{G^n(q)\}\to \{\bar q_1,\bar q_2\}.$
Since $G$ has no $2$-cycles, we must have $\bar p_1=\bar p_2=\bar p$ and $\bar q_1=\bar q_2=\bar q.$ Now, if $\bar p=\bar q,$ then $\zeta_G$ converges to $\bar p;$ otherwise, $\zeta_G$ is eventually bounded between $\bar p$ and $\bar q,$ i.e., $\bar p\leq_\tau \liminf \zeta_G\leq_\tau \limsup \zeta_G\leq_\tau \bar q.$
\end{proof}

We now arrive at the following main statement about the convergence of our original discrete system \eqref{recursive}. Recall that $x$ is a fixed point of $F$ if
$F(x,\ldots, x)=x$, and if an orbit of system \eqref{recursive} converges to $x$, then $x$ must be a fixed point of $F$.

\begin{theorem}\label{nested} (Nested Box Convergence)
Let $\alpha_F$ and $\beta_T$ as in Eqs. \eqref{recursive} and \eqref{Eq-mapT}. Suppose $G$ is a diagonal extension of $T$ to $V^k\times V^k$, and there is a partial ordering $\tau$ on $V^k$ so that $G$ is increasing on $V^k\times_\lambda V^k$. Assume there is $P,Q\in V^k\times_\lambda V^k$ so that $P\leq G(P), G(Q)\leq Q$. If $G$ has a unique fixed point in $[P,Q]$, then $G^n(Y_0)$ converges for all $Y_0=(X_0,X_0)$ that satisfy $P\leq Y_0\leq Q.$ In that case, \eqref{recursive}  converges to a fixed point of $F$.
\end{theorem}

\begin{proof} This is a direct application of Part (a) in Lemma \ref{squeeze}  to the sequence
$\zeta_G=\{G^n(X_0,X_0)\})_{n\geq 0}$ in $(V^k\times V^k, \tau\times\tau^\vee)$. For the initial value $Y_0 = (X_0,X_0)$, the sequence $G^n(Y_0)$ corresponds to the sequence $\beta_T\times\beta_T$. This sequence converges iff $\beta_T$ converges iff $\alpha_F$ converges.
\end{proof}

\subsection{Monotonic sequences under the new partial orders}

To achieve convergence for the given $\alpha_F$ with delay $k$, the game plan is to construct a partial order $\tau$ and an extension $G_\tau$ as described in Theorem \ref{nested}. This will lead to creating an appropriate trapping box for the discrete dynamical system mentioned in Eq. \eqref{recursive}.
Let $(V,\leq)$ and $(W,\preceq)$ be two partially ordered sets, and let $f: V\rightarrow W$ be a function. We say that $f$ is increasing with respect to the given partial orders if
$x\leq y\ \Longrightarrow\ f(x)\preceq f(y)$.
In particular, for $(V,\leq$), when $\tau$ is a chosen partial ordering on $V^k$, then $F: V^k\rightarrow V$ is an increasing function with respect to $\tau$ if for $X,Y\in V^k,$
\begin{equation}\label{incF}
X\leq_\tau Y\ \Longrightarrow\
F(X)\leq F(Y).
\end{equation}
Such a function is denoted by $F(\uparrow_\tau)$. If $\tau$ is the south-east order, that is, $\tau (1)=1$ and $\tau (2)=-1$, then $F(\uparrow_\tau),$ which is denoted $F(\uparrow,\downarrow)$ in the literature \cite{Al-Al-Ka2020, Al2022}.
\\

Given a monotone function $F: V^k\rightarrow V$ as above and its ``vector-form'' $T$ as given in Eq. \eqref{Eq-mapT}, we define a higher-dimensional system (and call it an embedding of the original system) as follows:

\begin{definition}\label{Def-Gtau} (Diagonal Embedding). \rm
Assume $F(\uparrow_\tau)$ on $(V^k,\tau)$, and let $X=(x_1,x_2,\ldots, x_k), U=(u_1,u_2,\ldots, u_k)\in V^k.$ Define
$G_\tau\; : V^k\times V^k\to V^k\times V^k$  as follows:
\begin{enumerate}
    \item If $\tau (1)=1$, then
$G_\tau (X,U) =
(F(X), y_2,\ldots, y_{k}, F(U), z_2,\ldots z_{k}),$
$$\hbox{where for $1\leq i\leq k-1$}\ , \ \begin{cases} y_{i+1}=x_i\ \hbox{and}\ z_{i+1}=u_i&\hbox{if}\ \tau(i)\tau (i+1)=1\\
y_{i+1}=u_i\ \hbox{and}\ z_{i+1}=x_i&\hbox{if}\ \tau(i)\tau (i+1)=-1.
\end{cases}$$
\item If $\tau (1)=-1$, then
$G_\tau (X,U) =
(F(U), y_2,\ldots, y_{k}, F(X), z_2,\ldots z_{k}),$
where the entries $2$ to $k$ are chosen similarly as in the first case.
In other words, one switches $x$'s and $u$'s when the function at that coordinate $i$ changes monotonicity.
\end{enumerate}
\end{definition}

\begin{remark}\rm It is an immediate observation that the diagonal of $V^{2k}$ is invariant under $G_\tau$, for any choice of $\tau$. In particular, $G_\tau$ is a  diagonal extension of $T$ in the sense that
$$G_\tau (x_1,x_2,\ldots, x_k, x_1,x_2,\ldots, x_k) =
(T(x_1,x_2,\ldots, x_k), T(x_1,x_2,\ldots, x_k)).$$
\end{remark}

The main reason for introducing the diagonal extension $G_\tau$ is the following proposition, which shows that $G_\tau$ is a suitable choice of extension to which Theorem \ref{nested} applies perfectly well.

\begin{proposition}\label{keyembedding} If $F(\uparrow_\tau)$, then
$G_\tau$ is increasing with respect to the $\lambda=\tau\times\tau^\vee$ partial ordering.
\end{proposition}

Before giving the proof, we illustrate and further clarify the construction of $G_\tau$.

\begin{example}\rm Below are examples of the map $F(\uparrow_\tau)$ and its associated extension $G_\tau.$
\begin{description}
    \item{(i)} Consider $F(\uparrow,\downarrow).$ Then $\tau : \{1,2\}\to \{-1,1\}$ is such that $\tau (1)=1$ and $ \tau(2)=-1$ (the southeast ordering again), and
    $$G_\tau (x_1,x_2,u_1,u_2) = (F(x_1,x_2),u_1, F(u_1,u_2), x_1).$$
    Notice the interchange between $u_1$ and $x_1$ because $\tau(1)\tau(2)=-1$.
 \item{(ii)} Consider $F(\uparrow,\downarrow,\uparrow).$ Then $\tau : \{1,2,3\}\to \{-1,1\}$ is such that $\tau (1)=1,\tau(2)=-1$ and $ \tau(3)=1$, and
    $$G_\tau (x_1,x_2,x_3,u_1,u_2,u_3) = (F(x_1,x_2, x_3),u_1,u_2, F(u_1,u_2,u_3), x_1,x_2).$$
    We switched entries between the block of $x$'s and the block of $u$'s twice because $\tau$ changed monotonicity twice (at those entries).
\item{(iii)} Consider $F(\downarrow,\downarrow,\uparrow).$ Define $\tau : \{1,2,3\}\to \{-1,1\}$ such that $\tau (1)=-1,\tau(2)=-1$ and $ \tau(3)=1.$   Then
    $$G_\tau (x_1,x_2,x_3,u_1,u_2,u_3) = (F(u_1,u_2,u_3),x_1, u_2, F(x_1,x_2, x_3), u_1,x_2).$$
\item{(iv)}  Consider $F(\uparrow,\uparrow,\downarrow).$ Define $\tau : \{1,2,3\}\to \{-1,1\}$ such that $\tau (1)=1,\tau(2)=1$ and $ \tau(3)=-1.$   Then
    $$G_\tau (x_1,x_2,x_3,u_1,u_2,u_3) = (F(x_1,x_2,x_3),x_1, u_2, F(u_1,u_2, u_3), u_1,x_2).$$
\end{description}
\end{example}

\vskip 5pt
\begin{proof} (of Proposition \ref{keyembedding}).
Let $X, U, Y, V\in V^k$, and assume $(X,U)\leq_{\lambda }(Y,V)$. The claim we must establish is that
$G(X,U)\leq_{\lambda } G(Y,V)$.
We write $X=(x_1,\ldots, x_k)$, $Y=(y_1,\ldots, y_k)$, $U=(u_1,\ldots, u_k)$ and $V = (v_1,\ldots, v_k)$. We start by unraveling the first inequality: $(X,U)\leq_{\lambda}(Y,V)$ means that
$X\leq_\tau Y$ and $V\leq_\tau U$, and thus more explicitly
\begin{equation}\label{signs}
\begin{cases} x_i\leq y_i\ \ \hbox{and}\ \ v_i\leq u_i,&\ \hbox{if}\ \tau (i)=1\\
x_i\geq y_i\ \ \hbox{and}\ v_i\geq u_i,& \ \hbox{if}\ \ \tau (i)=-1.
\end{cases}
\end{equation}
Notice that $x_i\leq y_i\ \Longleftrightarrow\ \ u_i\geq v_i$. What we need to show is that $G(X,U)\leq_\lambda G(Y,V).$ Based on Definition \ref{Def-Gtau}, there are two cases to consider.

\noindent\textbf{Case one:} $\tau(1)=1$.
In that case, we write
\begin{eqnarray}\label{blocks}
G(X,U) &=& (F(x_1,\ldots, x_k), a_2,\ldots, a_k, F(u_1,\ldots, u_k), a'_2,\ldots, a'_k)\nonumber\\
G(Y,V) &=& (F(y_1,\ldots, y_k), b_2,\ldots, b_k, F(v_1,\ldots, v_k), b'_2,\ldots, b'_k)
\end{eqnarray}
with entries as specified in Definition \ref{Def-Gtau}.
The inequality $G(X,U)\leq_\lambda G(Y,V)$ breaks down into two inequalities that we must check. We start with the first
\begin{eqnarray}\label{firstblock}
(F(x_1,\ldots, x_k), a_2,\ldots, a_k)&\leq_\tau&
(F(y_1,\ldots, y_k), b_2,\ldots, b_k).
\end{eqnarray}
Since $X\leq_\tau Y$, then
$F(X)\leq F(Y)$ by \eqref{incF} and this is compatible with $\tau (1)=1$.  We have to verify next that $a_i\leq b_i$ if $\tau (i)=1$ and $a_i\geq b_i$ if $\tau (i)=-1$. We know that these entries are of the form
$$(i)\ \ (a_i, b_i) = (x_{i-1}, y_{i-1})
\ \ \hbox{or}\ \
(ii)\ \ (a_i, b_i) = (u_{i-1}, v_{i-1}).
$$
depending on the alternation in the sign of $\tau$.
\begin{description}
\item{Subcase 1}: Suppose $\tau (i)=1$. Then either $(a_i, b_i) = (x_{i-1}, y_{i-1})$ when
$\tau (i-1)=1$ (by the definition of $G$), in which case $a_i\leq b_i$
since $x_{i-1}\leq y_{i-1}$,
or $\tau (i-1)=-1$ and $(a_i, b_i) = (u_{i-1}, v_{i-1})$, in which case also $a_i\leq b_i$ since
$x_{i-1}\geq y_{i-1}$ and thus $u_{i-1}\leq v_{i-1}$ according to \eqref{signs}.

\item{Subcase 2}: Suppose $\tau (i)=-1$, then also from the construction of $G$, we must make sure that $a_i\geq b_i$. Here too, two cases can occur : if $\tau (i-1)=-1$, then  $(a_i, b_i) = (x_{i-1}, y_{i-1})$ and   $x_{i-1}\geq y_{i-1}$ as desired, or (ii) $\tau (i-1)=1$,
and $(a_i, b_i) = (u_{i-1}, v_{i-1})$,
in which case $u_{i-1}\geq v_{i-1}$ according to \eqref{signs}, and so $a_i\geq b_i$ as well.
\end{description}
Hence, we established the inequality in \eqref{firstblock}. We need to show next that for the second block in \eqref{blocks}, the ordering is also satisfied, i.e.
\begin{equation}\label{secondblock}
(F(v_1,\ldots, v_k), b'_2,\ldots, b'_k)\leq_\tau (F(u_1,\ldots, u_k), a'_2,\ldots, a'_k).
\end{equation}
Here, $\tau(1)=1$ as we set out from the beginning, so we need to make sure that $F(U)\geq F(V)$. But this is ensured by \eqref{incF} since $U\geq V$. Checking the other entries is similar and thus omitted.

\noindent\textbf{Case two:} $\tau(1)=-1$. This case proceeds in the same manner as case one.
\end{proof}
\subsection{Fixed points and global attractors}

We clarify the fixed points of $G$ and their relationship with $F.$
As before, $F: V^k\rightarrow V$ is increasing with respect to  $\tau : \{1,2\ldots,k\}\rightarrow \{-1,1\}$ as in Definition \ref{Def-Poset}, and
$\alpha_F: x_{n+1}=F(x_n,\ldots, x_{n-k+1})$.
Define the point (depending on $x,y$)
\begin{equation}\label{Eq-Ptau}
P_\tau = (x_1,\ldots, x_k) \quad \text{such that}\quad x_i = \begin{cases} x,&\ \hbox{if}\ \tau (i)=1\\
y,&\ \hbox{if}\ \tau (i)=-1,
\end{cases}
\end{equation}
and define $P^t_\tau$ its ``dual'' obtained by replacing $x$'s by $y$'s. Notice that, by construction, if $x < y$ then $P_\tau <_\tau P_\tau^t$.

\begin{lemma}\label{Eq-FixedPointsInGeneral}
The fixed points of $G_\tau$ are the points $P_\tau\times P_\tau^t$ satisfying
\begin{equation}
\begin{cases} F(P_\tau) = x\ ,\ F(P^t_\tau) =y\ ,\ \hbox{if}\ \tau (1)=1\\
F(P_\tau) = y\ ,\ F(P^t_\tau) =x\ ,\ \hbox{if}\ \tau (1)=-1.
\end{cases}
\end{equation}
\end{lemma}

\begin{proof} Write $\xi = (x_1,\ldots, x_k, u_1,\ldots, u_k)\in V^k\times V^k$. We solve for
$G(\xi)=\xi$. The argument will show that $P_\tau = (x_1,\ldots, x_k)$ and $P_\tau^t = (u_1,\ldots, u_k)$.
As before, the proof breaks down into two cases: $\tau (1)=1$ and $\tau (1)=-1$. We can first assume that $\tau (1)=1$. In this case,
$G(\xi)=\xi$ equates to
\begin{equation}\label{gzeta}
(F(x_1,\ldots, x_k), y_2,\ldots , y_{k}, F(u_1,\ldots, u_k), z_2\ldots, z_{k}) = (x_1,\ldots, x_k, u_1,\ldots, u_k)
\end{equation}
where $y_i$ is either $x_{i-1}$ or $u_{i-1}$, and $z_j$ is either $x_{j-1}$ or $u_{j-1}$, as stipulated in
Definition \ref{Def-Gtau}.
We set $x=x_1$ and $y=u_1$.
The point of the proof is that \eqref{gzeta} will produce two sequences
$x=c_2=\cdots = c_{k-1}$, where $c_i$ is either $x_i$ or $u_i$, and a complementary sequence
$y = d_1 = \cdots = d_{k-1}$, where
similarly, $d_i$ is either $x_i$ or $u_i$. The point $P_\tau = (x_1,\ldots, x_k)$ is obtained by replacing all of its entries in terms of $x$ or $y$. Moreover, the fact that the sequences $(c_i)$ and $(d_i)$ are complementary, meaning that if $c_i=x_i$ then $d_i=u_i$, and vice-versa, implies that $(u_1,\ldots, u_k)=P_\tau^t$.
It follows that $\xi=(P_\tau,P_\tau^t),$ and $(F(P_\tau),F(P^t_\tau))=(x,y)$ which is what we wanted to prove. The case $\tau (1)=-1$ is similar.
\end{proof}

\begin{remark} (Pseudo fixed points). \rm If $(P_\tau, P_\tau^t)$ is a fixed point of $G$, then so is $(P_\tau^t, P_\tau)$.
If $x=y$, $P_\tau = (x,\ldots, x)$ and the fixed points of $G$ are situated along the diagonal of $V^{2k}$. In this case, $x$ is a fixed point of $F.$ Fixed points $(P_\tau, P_\tau^t)$ and $(P_\tau^t, P_\tau)$ for $x\neq y$, are called \textit{pseudo or artificial} fixed points of $F.$ Those points come in pairs. In particular, if $G$ has a unique fixed point, then $x=y$ and $x$ is a fixed point of $F$.
\end{remark}
Finding the fixed points of the extension map $G$ consistently necessitates solving two equations with two unknowns, as seen in the subsequent example.
\begin{example}\phantom{.}\rm  We consider some diagonal extensions $G_\tau$ and describe their fixed points.
\begin{itemize}
\item Assume that $F(\uparrow,\uparrow,\downarrow).$ Then $P_\tau = (x,x,y)$ and solving $(F(P_\tau),F(P_\tau^t))=(x,y)$ is equivalent to solving the system $(F(x,x,y), F(y,y,x)) = (x,y)$ in $x$ and $y$. If $x=y,$ then we obtain the fixed points of $F;$ otherwise, we obtain the pseudo-fixed points of $F.$
\item Let $F(\downarrow, \uparrow,\uparrow,\downarrow).$ In this case, we have $P_\tau(y,x,x,y),$ and we find the fixed points of $G$ by solving $(F(y,x,x,y), F(x,y,y,x))=(y,x)$.
\end{itemize}
\end{example}

Recall that the set of limit points of an orbit $\mathcal{O}^+_F(X_0)=\{x_{n}\}_{n\geq -k}$ is called the omega limit set of $X_0$, and we denote it by $\omega(X_0)$. We now state a stability result that gives a unified and general statement of various known results in the literature \cite{Sm2006, Sm2008,Ku-Me2006,Al2022}.  This should be viewed as a special case of Theorem \ref{nested}.

\begin{theorem}\label{Th-GlobalStability} Consider Eq. \eqref{Eq-GeneralCase1} in which $V=[a,b].$ Define $\tau$ so that $F$ is increasing with respect to $\tau.$ If $F$ has no pseudo-fixed points, then for each $X_0\in V^k,$ $\omega(X_0)$ is a fixed point of $F.$ In particular, if $F$ has a unique fixed point $x$, then $x$ is a global attractor.
\end{theorem}

\begin{proof} Endow $V^k$ with the order $\leq_\tau$ and let $G_\tau$ be the diagonal extension (we will drop the lower index $\tau$ from notation for simplicity). Define $m_\tau$ and $M_\tau\in V^k$ as follows: The $i$-th entry of $m_\tau$ is $a$ if $\tau (i)=1$, and it is $b$ otherwise. The point $M_\tau$ is the ``dual'' of $m_\tau$ obtained by switching the $a$'s and $b$'s. These points are the extreme points for the order $\tau$, that is, for every
$\zeta\in V^k$, $m_\tau\leq \zeta\leq M_\tau$.  As a result, the pairs $(m_\tau, M_\tau)$ and $(M_\tau, m_\tau)$ are extreme points for the order $\lambda$ on $V^k\times V^k$, and necessarily
$$(m_\tau,M_\tau)\leq_\lambda G(m_\tau,M_\tau)\leq_\lambda G(M_\tau,m_\tau)\leq_\lambda (M_\tau,m_\tau)$$
so we always get a trapping box  in $V^k\times V^k$. Since
$G$ is increasing with respect to $\lambda$ by Proposition \ref{keyembedding}, we can apply Theorem \ref{nested} to obtain that $G^n (\zeta )$ converges to a fixed point of $G$, say $\bar \zeta$. For $X_0\in V^k$, set $\zeta = (X_0,X_0)$.
The assumption that $F$ has no pseudo-fixed points leaves
$G$ with fixed points $\bar \zeta = (x,\ldots, x)$ that are coming from fixed points of $F.$ Therefore, $\omega (X_0)$ is a fixed point of $F$. If this fixed point is unique, it must be globally attracting.
\end{proof}

If $V=\bbr^+$, the statement of Theorem \ref{Th-GlobalStability} can be rephrased in the following corollary:

\begin{corollary}
Let $V=\bbr^+$, $F: V^k\rightarrow V$ with $F(\uparrow_\tau)$ and $P_\tau$ as defined in \eqref{Eq-Ptau}. Suppose that $F$ has no pseudo-fixed points. If for each $X_0\in V^k,$ there exist $x<y$ such that $P_\tau<_\tau X_0<_\tau P^t_\tau$ and $(x,y)<_\lambda (F(P_\tau),F(P^t_\tau)),$ then $\omega(X_0)$ is a fixed point of $F.$ In particular, if $F$ has a unique fixed point, then it is a global attractor.
\end{corollary}
\begin{proof}
     The proof is analogous to that of Theorem \ref{Th-GlobalStability} with the following new ingredients:  Firstly, $x<y$ implies that $P_\tau<_\tau P_\tau^t$ \eqref{Eq-Ptau}. Secondly, $(x,y)<_\lambda (F(P_\tau),F(P^t_\tau))$ implies, by construction of the diagonal extension $G$,  that $(P_\tau, P_\tau^t ) <_\lambda G(P_\tau, P_\tau^t )$ and
    $G(P_\tau^t, P_\tau) <_\lambda (P_\tau^t, P_\tau)$ so that the pairs
     $(P_\tau,P_\tau^t)$ and $(P_\tau^t, P_\tau)$ provide the data of a trapping box. We can now apply Theorem \ref{nested} to obtain the fixed point globally attracting with respect to the trapping box. Since this process can be done for each $X_0\in V^k,$ we obtain the fixed point globally attracting with respect to $V.$
\end{proof}

\section{Embedding periodic systems }

In this section, we generalize the results of the previous section to non-autonomous $p$-periodic difference equations of the form
\begin{equation}\label{Eq-GeneralCase2}
\alpha : x_{n+1}=F_n(x_n,x_{n-1},\ldots,x_{n-k+1}),\quad k\geq 1,\; x_0,\ldots x_{1-k} \in \mathbb{R}_+,
\end{equation}
where each map $F_j$ is continuous and monotonic in each one of its arguments, $F_{n+p}=F_n$ for all $n=0,1,\ldots$ and $p$ is the minimal positive integer. Also, $X_{n+1}=T_n(X_n)$ can be used to represent the vector form of Eq. \eqref{Eq-GeneralCase2}, with $T_n$ defined as in Eq. \eqref{Eq-mapT}. To stress the role of the individual maps in the system, we can represent this periodic system by $[F_0,F_1,\ldots,F_{p-1}]$ or $[T_0,T_1,\ldots,T_{p-1}].$ $Fix([F_0,F_1,\ldots,F_{p-1}])$ denotes the fixed points of Eq. \eqref{Eq-GeneralCase2} and $Per([F_0,F_1,\ldots,F_{p-1}])$ denotes the periodic solutions. Also, for convenience, we can write $T_{i,j}=T_j\circ \cdots\circ T_i$ for all $i<j.$ As before, based on the monotonicity in the maps $F_j,$ we define the $\leq_\tau$ partial order, then define the associated extension $G_j$ for each $F_j.$ So, we obtain a $p$-periodic system
\begin{equation}\label{Eq-PeriodicGn}
\xi_{n+1}=G_n(\xi_n),\quad G_{n+p}=G_n\quad\text{and}\quad G_n\;:\; \mathbb{R}_+^k\times \mathbb{R}_+^k\to \mathbb{R}_+^k\times \mathbb{R}_+^k.
\end{equation}
The case $p=2$ was considered in \cite{Al2022} when $k=2.$ Define $\Phi_{i,j}:=G_{j}\circ\cdots\circ G_i,\; i<j$ then the fixed points of $\Phi_{0,p-1},$ i.e., $\Phi_{0,p-1}(\xi)=\xi$ play a crucial role in determining the attractors of Eq. \eqref{Eq-PeriodicGn}. As in the one-dimensional case, define $\mathcal{A}_{p,1}$ to be the divisors of $p$. Then a fixed point of $\Phi_{0,p-1}$ gives a $q$-cycle of Eq. \eqref{Eq-PeriodicGn}, for some $q\in\mathcal A_{p,1}$. A $q$-cycle of \eqref{Eq-PeriodicGn} that is not a $q$-cycle of $\alpha_{0,p-1}$ is called a pseudo (or artificial) $q$-cycle of Eq. \eqref{Eq-GeneralCase2}. However, if the solution of $\Phi_{0,p-1}(\xi)=\xi$ is unique, then it must give a $q$-cycle of Eq. \eqref{Eq-GeneralCase2} for some $q\in \mathcal{A}_{p,1}$ as we clarify next.

\begin{lemma}\label{Lem-Jose}
For $i=0,\ldots, p-1,$ let $T_i\;:\;U\;\to \; U$ and $G_i\;:\;V\;\to \; V$ be continuous maps. Assume $\phi\;:\; U\;\to\; V$ is an injective map that satisfies $\phi\circ T_i=G_i\circ \phi,$ then each of the following holds true:
\begin{description}
\item{(i)} If $[u_0,u_1,\ldots,u_{q-1}]$ is a $q$-cycle of $[T_{0},\ldots,T_{p-1}],$ then $[\phi(u_0),\phi(u_1),\ldots,\phi(u_{q-1})]$ is a $q$-cycle of $[G_{0},\ldots,G_{p-1}].$
\item{(ii)} If $[v_0,v_1,\ldots,v_{q-1}]$ is a $q$-cycle of $[G_{0},\ldots,G_{p-1}]$ and there exists a point $u_0\in U$ such that $\phi(u_0)=v_0,$ then the iteration of $u_0$ in $[T_{0},\ldots,T_{p-1}]$  gives a $q$-cycle.
\end{description}
\end{lemma}
\begin{proof}
Part (i) is straightforward because $\phi$ is an injective. So, we verify Part (ii). Assume $[v_0,v_1,\ldots,v_{q-1}]$ is a $q$-cycle of $[G_{0},\ldots,G_{p-1}],$ and there exists $u_0\in U$ such that $\phi(u_0)=v_0.$ Now, the fact that
\begin{align*}
G_n\circ\cdots\circ G_0(v_0)=&G_n\circ\cdots\circ G_0(\phi(u_0))
= \phi(T_n\circ\cdots\circ T_0(u_0))
\end{align*}
and the injective nature of $\phi$ forces the iterates of $u_0$ under $[T_{0},\ldots,T_{p-1}]$ to form a $q$-cycle.
\end{proof}
\begin{proposition}\label{Pr-UniqueCycle}
For each $j=0,\ldots, p-1,$ let $G_j$ be the extension of $F_j,$ and let $\Phi_{i,j}=G_{j}\circ\cdots\circ G_i.$ Each of the following holds true:
\begin{description}
\item{(i)} If Eq. \eqref{Eq-GeneralCase2} has a $q$-cycle for some $q\in\mathcal{A}_{p,1},$ then $\Phi_{0,p-1}$ has a fixed point.
\item{(ii)} If $\Phi_{0,p-1}$ has a unique fixed point, then Eq. \eqref{Eq-GeneralCase2} has a unique $q$-cycle for some $q\in\mathcal{A}_{p,1}.$
\end{description}
\end{proposition}
\begin{proof}
The proof depends on Lemma \ref{Lem-Jose}. There is a $q$-cycle $[x_0,\ldots,x_{q-1}]$ of Eq. \eqref{Eq-GeneralCase2} iff there is a $q$-cycle $[u_0,\ldots,u_{q-1}]$ of $X_{n+1}=T_n(X_n).$ Note that we have the following relationship between the maps:
\[\begin{tikzcd}
    \mathbb{R}_+^k\ar{r}{T_i}\ar{d}[swap]{\phi} & \mathbb{R}_+^k\ar{d}{\phi} \\
    \mathbb{R}_+^k\times\mathbb{R}_+^k\ar{r}[swap]{G_i} & \mathbb{R}_+^k\times \mathbb{R}_+^k
\end{tikzcd}\]
where the injective map $\phi$ is defined by $\phi(X)=(X,X).$ By Part (i) of Lemma \ref{Lem-Jose}, $[\phi(u_0),\ldots,\phi(u_q)]$ is a $q$-cycle of $\xi_{n+1}=G_n(\xi_n).$  Since $q$ divides $p,$  $\phi(u_0)$ is a fixed point of $\Phi_{0,p-1}.$ To verify Part (ii), suppose $\bar \xi$ is a unique fixed point of $\Phi_{0,p-1}.$ If $\bar \xi$ is not along the diagonal of $\mathbb{R}_+^k\times \mathbb{R}_+^k,$ then $\bar \xi^t$ must be another fixed point, which contradicts uniqueness. Therefore, $\bar \xi^t$ belongs to the diagonal, and the iterates of $\bar \xi^t$ under $\xi_{n+1}=G_n(\xi_n)$ give a $q$-cycle for some $q$ that divides $p.$ Now,  consider the pre-image of this $q$-cyle under $\phi.$ We obtain a $q$-cycle of $X_{n+1}=T_n(X_n).$ Hence, we obtain a $q$-cycle of Eq. \eqref{Eq-GeneralCase2}. The uniqueness of the cycle is obvious, and the proof is complete.
\end{proof}
Now, we are ready to give the analog of Theorem \ref{Th-GlobalStability} for periodic systems.
\begin{theorem}\label{Th-GlobalStability2} Consider the $p$-periodic system in Eq. \eqref{Eq-GeneralCase2}, where $F_j: \bbr_+^k\rightarrow \bbr_+$ and $F_j(\uparrow_\tau)$ for each $j$. Let $P_\tau$ as defined in Eq. \eqref{Eq-Ptau}, and suppose that Eq. \eqref{Eq-GeneralCase2} has a unique $q$-cycle with no pseudo-cycles. If for each $X_0\in \bbr_+^k,$ there exists $x<y$ such that $P_\tau<_\tau X_0<_\tau P^t_\tau$ and $(x,y)<_\lambda (F(P_\tau),F(P^t_\tau)),$ then $\omega(X_0)=Per([F_0,\ldots,F_{p-1}]).$  In particular, Eq. \eqref{Eq-GeneralCase2} has a global attracting $q$-cycle for some $q\in\mathcal{A}_{p,1}$.
\end{theorem}
\begin{proof}
Embed the $p$-periodic system in Eq. \eqref{Eq-GeneralCase2} into a $p$-periodic $2k$-dimensional system $\xi_{n+1}=G_n(\xi_n)$ as in Eq. \eqref{Eq-PeriodicGn}, where each extension $G_n$ is defined as in Definition \ref{Def-Gtau}. The map $\Phi_{0,p-1}=G_{p-1}\circ\cdots\circ G_0$ maps the closed interval
$[(P_\tau, P_\tau^t), (P_\tau^t,P_\tau)]$ (in the $\lambda$ ordering of $\mathbb{R}^k_+\times\mathbb{R}^k_+$) into itself, and consequently, it has a fixed point in that interval. This fixed point contributes to a $q$-cycle of the $p$-periodic system $[G_0,\ldots,G_{p-1}],$ and $q$ must be a divisor of $p$ as clarified in Proposition \ref{Pr-UniqueCycle}. The uniqueness assumption on the $q$-cycle of Eq. \eqref{Eq-GeneralCase2} and the absence of pseudo cycles make the $q$-cycle of  $[G_0,\ldots,G_{p-1}]$ unique. Some elements of this $q$-cycle serve as fixed points of $\Phi_{0,p-1},$ and consequently, each orbit of $\Phi_{0,p-1}$ converges to a fixed point. Therefore, the $q$-cycle of $[G_0,\ldots,G_{p-1}]$ is a global attractor, which leads to the $q$-cycle of Eq. \eqref{Eq-GeneralCase2} being a global attractor.
\end{proof}
Our next example illustrates the relationship between cycles of Eq. \eqref{Eq-PeriodicGn} and Eq. \eqref{Eq-GeneralCase2}.

\begin{example}\rm
(i)
Consider the $2$-periodic system
$$x_{n+1}=F_n(x_n,x_{n-1},x_{n-2})=\frac{bx_n}{1+x_{n-2}}+h_{n\bmod 2},\quad n=0,1,\ldots$$
where $h_0=1,h_1=3$ and $b=4.$ Define $T_j(x,y,z)=(F_j(x,y,z),x,y),$ then solve $T_1(T_0(X))=X=(x,y,z).$ In this case, we obtain the $2$-periodic solution $\{\bar x,\bar y\},$ where $\bar x=3+\frac{4\sqrt{6}}{3}$ and $\bar y=2+\sqrt{6}.$ On the other hand, the equation $\xi_{n+1}=G_n(\xi_n)$ has the unique $2$-cycle $\{\bar \eta_1,\bar \eta_2\},$ where $\bar \eta_1=(\bar X,\bar X),$ $\bar \eta_2=(\bar Y,\bar Y)$ and $\bar X=(\bar x,\bar y,\bar x)$ and $\bar Y=(\bar y,\bar x,\bar y).$
\\

\noindent (ii) consider the $2$-periodic system
\begin{equation}\label{Eq-2PeriodicFn}
x_{n+1}=F_n(x_n,x_{n-1},x_{n-2})=\frac{bx_n}{1+x_{n-2}^2}+h_{n\bmod 2},\quad n=0,1,\ldots
\end{equation}
where $h_0=1.8,h_1=2.3$ and $b=4.$ Then consider the embedded system
\begin{equation}\label{Eq-2PeriodicGn}
\xi_{n+1}=G_n(\xi_n),\quad G_j(x_1,x_2,x_3,u_1,u_2,u_3)=(F_j(x_1,x_2,x_3),x_1,u_2,F_j(u_1,u_2,u_3),u_1,x_2).
\end{equation}
 Let $\bar x\approx 3.55$ and $\bar y\approx 2.84.$ Eq. \eqref{Eq-2PeriodicFn} has the $2$-cycle $\{\bar x,\bar y\},$ which contributes to the $2$-cycle
 $$\{(\bar X_1,\bar X_1),(\bar Y_1,\bar Y_1)\},\quad \bar X_1=(\bar x,\bar y,\bar x)\;\;\text{and}\;\; \bar Y_1={\bar X}_1^t$$
 of Eq. \eqref{Eq-2PeriodicGn}.
 However, Eq. \eqref{Eq-2PeriodicGn} has two more $2$-cycles, namely
 $\{\bar\xi,G_0(\bar \xi)\}$ and $\{\bar\eta,G[0](\bar\eta)\}$,
 where $\bar \xi=(a_1,a_2,b_1,b_1,b_2,a_1)$, $ a_1\approx 2.82$, $b_1\approx 4.99$,
 $a_2=\frac{ba_1}{1+b_1^2}+h_0\approx 2.24$,  $b_2=\frac{bb_1}{1+a_1^2}+h_0\approx 4.03$ and  $\bar \eta=(b_1,b_2,a_1,a_1,a_2,b_1).$
 \end{example}
\section{Applications}
In this section, we provide two illustrative examples that demonstrate the effectiveness of our constructed theoretical framework in addressing global stability. The first example is a mathematical model, while the second covers a broad class of rational difference equations.

\subsection{The Ricker model with delays and stocking}
Consider the Ricker model with delays in recruitment and constant stocking
\begin{equation}\label{Eq-RickerExample}
x_{n+1}=F(x_n,x_{n-1},\ldots,x_{n-k+1})=x_n\exp(r-x_{n-k})+h, r,h>0, x_0,x_{-1},\ldots, x_{-k+1}\geq 0.
\end{equation}
This equation has a unique positive equilibrium $\bar x_h$, which must be larger than $h.$ This model was previously investigated in \cite{Al-Ka2023}, specifically with $k$ set to $1.$ In this example, we aim to examine the general options for the delay $k$, apply our established theoretical framework to attain global stability, and then compare the results with local stability. As observed in the sequel, demonstrating local stability for $k>2$ presents a challenging task.  We begin by the well-known case $k=0,$ we obtain local and global stability when
\begin{align}\label{In-k=0}
r<r_0:=&\bar x_h +\ln\left(1-\frac{h}{\bar x_h}\right)\nonumber \\
=&\frac{1}{2}\left(2+h+\sqrt{h^2+4}\right)-\ln\left(\frac{2-h + \sqrt{h^2 + 4}}{2+h + \sqrt{h^2 + 4}}\right).
\end{align}
When $k=1,$ we obtain local stability \cite{Al-Ka2023} when
\begin{equation}\label{In-k=1}
r<r_1:=h+1-\ln(h+1).
\end{equation}
When $k=2,$ the characteristic polynomials is
$$p(t)=t^3-\left(1-\frac{h}{\bar x}\right)t^2+\bar x-h.$$
From the Jury's necessary conditions, we need $p(0)=\bar x-h<1,$ $p(1)>0$ and $-p(-1)>0.$ $p(0)<1$ gives us Condition \eqref{In-k=1} that we obtain when $k=1.$    The condition $p(1)>0$ is valid by default. The condition
$-p(-1)>0$ implies $-\bar x_h^2 + (h + 2)\bar x_h - h>0,$ and consequently
$$\bar x_h <1 +\frac{h}{2}+ \frac{1}{2}\sqrt{h^2 + 4)}.$$
This condition is, in fact, the same as $r<r_0$ in \eqref{In-k=0}. From the Jury's sufficient conditions, we need
$$p(\bar x_h)=\bar x_h(\bar x_h-h+1)(\bar x_h-h-1)-(\bar x_h-h)^2>0.$$
Since $p(-2)=h^2 + 4h + 2>0,$ $p(0)=-h^2<0$ and $p(h)=h,$ the intermediate value theorem tells us there are three real roots, and only one of them is larger than $h.$ Therefore, $\bar x_h$ loses its stability when it reaches the largest zero of $p(x).$
This value can be found explicitly; however, we avoid its formidable expression here and present the curve in Fig. \ref{Fig-StabilityRegions1}. When $k=3,$ we do the computations numerically and again give the curve in Fig. \ref{Fig-StabilityRegions1}. As the value of $k$ increases, the task of proving local stability becomes increasingly difficult.
\\

Next, we turn our attention to global stability based on Theorem \ref{Th-GlobalStability}. We have $F(x_1,\ldots,x_{k})=x_1e^{r-x_k}+h$, and consequentially, we consider $F(\uparrow,\uparrow,\cdots, \uparrow,\downarrow).$ Therefore, solving $F(x_1,\ldots,x_{k})=x_1e^{r-x_k}+h$ leads to solving $F(x,x,\ldots, x,y)=x$. Define $\leq_{\tau}$ and $G$ as in Definition \ref{Def-Gtau}. First, we investigate the solution of $\xi=G(\xi).$ This gives us
$$F(x,\ldots,x,y)=xe^{r-y}+h=x\quad \text{and}\quad F(y,\ldots,y,x)=ye^{r-x}+h=y.$$
Therefore, we obtain the same result as in the case $k=1$ \cite{Al-Ka2023}. In particular, the equilibrium solution $(x,y)=(\bar x_h,\bar x_h)$ is the unique solution as long as
\begin{equation}\label{Eq-r-infinity}
r<r_\infty:=\frac{1}{2}\left(h+\sqrt{4h+h^2}\right)+\ln\left(1-\frac{2h}{h+\sqrt{4h+h^2}}\right).
\end{equation}
Again, we plot the curve $r=r_\infty$ in Fig. \ref{Fig-StabilityRegions1}.
At $r=r_\infty,$ two new solutions emerge, which are denoted by pseudo-fixed points. We give the conclusion in the following result:
\begin{corollary}
Consider Eq. \eqref{Eq-RickerExample} with $h>0,$ and let $r_\infty$ as defined in Eq. \eqref{Eq-r-infinity}. If $r<r_\infty,$ then the equilibrium $\bar x_h$ is a global attractor for any finite delay $k.$
\end{corollary}
\begin{proof}
When $r<r_\infty,$ the equation $G(\xi)=\xi$ has a unique solution. Let $a<b$ and consider $A=(a,\ldots,a,b)$ while $B$ has the components of $A$ switched, i.e., $a\leftrightarrow b.$ We have
$(A,B)<_{\lambda}G(A,B)$ iff
$$a<ae^{r-b}+h\quad \text{and}\quad be^{r-a}+h<b.$$
Let the set of feasible solutions be $\Omega.$ $\Omega$ is not empty when $r<r_\infty$ (cf. \cite{Al-Ka2023}). Furthermore, for each initial condition $X_0=(x_0,\ldots,x_{-k})$ there exists $(a,b)\in\Omega$ such that
$(A,B)<_{\lambda}(X_0,X_0)<_{\lambda}(B,A)$ and $(A,B)<_{\lambda}G(A,B).$ By Theorem \ref{Th-GlobalStability}, $\bar x_h$ is a global attractor of Eq. \eqref{Eq-RickerExample}.
\end{proof}
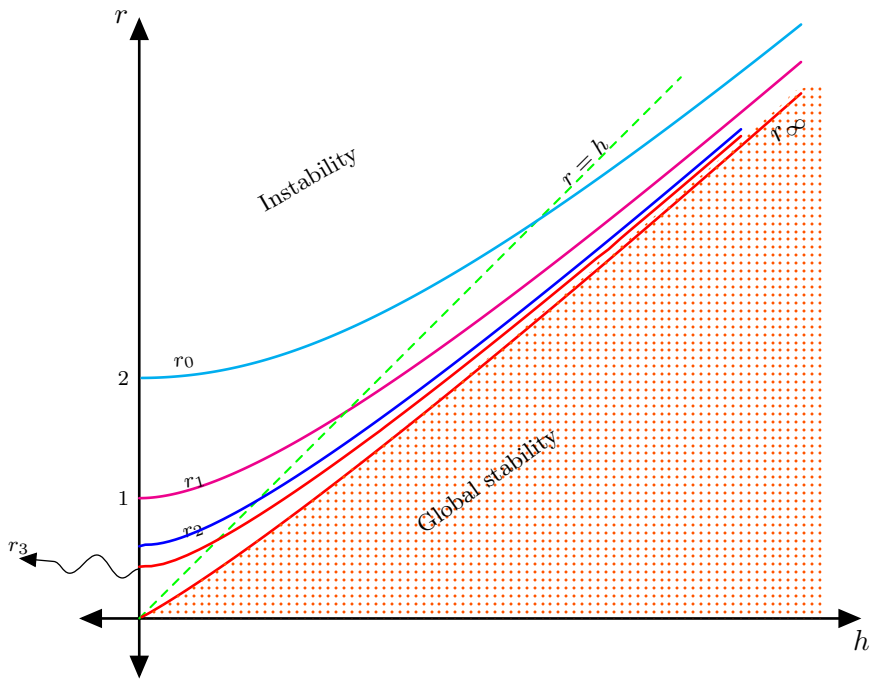
\begin{figure}[H]
\begin{center}
\definecolor{MyMaroon}{rgb}{128,0,0}
\definecolor{MyOrange}{rgb}{255,87,51}
\definecolor{ffvvqq}{rgb}{1,0.3333333333333333,0}
\begin{tikzpicture}[line cap=round,line join=round,>=triangle 45,x=1.0cm,y=1.0cm,scale=1.6]
\draw[-triangle 45, line width=1.0pt,scale=1] (0,0) -- (6.0,0) node[below] {$h$};
\draw[line width=1.0pt,-triangle 45] (0,0) -- (-0.5,0);
\draw[-triangle 45, line width=1.0pt,scale=1] (0,0) -- (0,5.0) node[left] {$r$};
\draw[line width=1.0pt,-triangle 45] (0,0) -- (0.0,-0.5);
\draw[line width=1pt,color=magenta,smooth,samples=100,domain=0:5.5] plot(\x,{\x+1-ln(\x+1)});
\draw[line width=0.8pt,color=green,dashed,domain=0:4.5] plot(\x,{\x});
\draw[line width=1.0pt,color=cyan,smooth,samples=100,domain=0.02:5.5] plot(\x,{ln((\x*\x+4)^0.5-\x+2)-ln(\x++2+(\x*\x+4)^0.5)+\x/2+(\x*\x+4)^0.5/2+1});
\draw[line width=1pt,color=MyMaroon,smooth,samples=100,domain=0.02:5.5] plot(\x,{ln((\x*\x+4*\x)^0.5-\x)-ln(\x+(\x*\x+4*\x)^0.5)+\x/2+(\x*\x+4*\x)^0.5/2});
\draw[scale=1] (0.0,2.0) node[left,rotate=0] {\scriptsize $2$};
\draw[scale=1] (0.4,2.0) node[above,rotate=10] {\scriptsize $r_0$};
\draw[scale=1] (0.5,1.0) node[above,rotate=15] {\scriptsize $r_1$};
\draw[scale=1] (0.5,0.6) node[above,rotate=20] {\scriptsize $r_2$};
\draw[line width=0.5pt,-triangle 45,decorate, decoration={
    snake,
    segment length=20,
    amplitude=3.9,post=lineto,
    post length=10pt
}]  (0.0,0.41) -- (-1,0.5);
\draw[scale=1] (-1,0.45) node[above,rotate=0] {\scriptsize $r_3$};
\draw[line width=1.0pt,color=blue,smooth]
(0.0,0.6)--(0.05,0.613198)--(0.1,0.614348)--(0.15,0.619992)--(0.2,0.629162)--(0.25,0.641194)--(0.3,0.655605)--(0.35,0.672036)--(0.4, 0.690209)--(0.5,0.730946)--(0.6,0.77651)--(0.7,0.826004)--(0.8,0.878782)--(0.85,0.906245)--(0.95, 0.963071)--
(1.,0.992351)--(1.1,1.05247)--(1.2,1.11447)--(1.3,1.17815)--(1.4,1.24335)--(1.5,1.30993)--(1.6,1.37777)--(1.7,1.44676)--(1.75,1.48166)--(1.8,1.51682)--(1.9,1.58786)--(1.95,1.62373)--(2.,1.65982)--(2.1,1.73263)--(2.15,1.76933)--(2.2,1.80623)--
(2.3,1.88059)--(2.35,1.91803)--(2.4,1.95565)--(2.5,2.03137)--(2.6,2.10771)--(2.7,2.18465)--(2.8,2.26214)--(2.9,2.34017)--(3.,2.4187)--(3.1,2.49772)--(3.2,2.57719)--(3.3,2.6571)--(3.4,2.73743)--(3.5,2.81816)--(3.6,2.89927)--(3.7,2.98075)--(3.8,3.06259)--
(3.9,3.14476)--(3.95,3.18597)--(4.,3.22726)--(4.1,3.31008)--(4.15,3.3516)--(4.2,3.39319)--(4.3,3.4766)--(4.35,3.51841)--(4.4,3.56029)--(4.45,3.60224)--(4.5,3.64425)--(4.6,3.72848)--(4.7,3.81295)--(4.8,3.89768)--(4.9,3.98264)--(5.,4.06783);
\draw[line width=1.0pt,color=red,smooth]
(0,0.43)--(0.05,0.434366)--(0.1,0.434642)--(0.2,0.453133)--(0.3,0.485607)--(0.4,0.526698)--(0.5,0.57379)--(0.6,0.625383)--(0.7,0.680525)--(0.8,0.738569)--(0.9,0.799049)--(1.,0.861615)--(1.1,0.925998)--(1.2,0.991986)--(1.3,1.0594)--(1.4,1.12811)--(1.5, 1.19798)--(1.6,1.26893)--(1.7,1.34085)--(1.8,1.41367)--(1.9,1.48734)--(2.,1.56)--(2.05,1.599)--(2.1,1.63695)--(2.2,1.71281)--(2.3,1.7893)--(2.4,1.8664)--(2.45,1.90517)--(2.5,1.94407)--(2.6,2.02228)--(2.7,2.101)--(2.8,2.1802)--(2.9,2.25987)--
(3.,2.33997)--(3.1,2.42049)--(3.2,2.50142)--(3.3,2.58273)--(3.4,2.6644)--(3.5,2.74642)--(3.6,2.82879)--(3.7,2.91147)--(3.75,2.95294)--(3.8,2.99447)--(3.9,3.07)--(4.,3.16)--(4.1,3.24523)--(4.2,3.32936)--(4.3,3.41376)--(4.4,3.4984)--(4.45,3.54082)--(4.5,3.58329)--
(4.6,3.66841)--(4.7,3.75376)--(4.8,3.83933)--(4.9,3.92511)--(5.,4.0111);
\fill[dotted,scale=0.9,color=ffvvqq,fill=ffvvqq,pattern=dots,pattern color=ffvvqq] (6.0,4.86)--(6.0,4.83)--(5.99,4.80)--(5.955,4.7699)--(5.92,4.7389)--(5.88,4.7078)--(5.85,4.6768)--(5.8152,4.6458)--
(5.78016,4.6148)--(5.745,4.5838)--(5.71,4.552)--(5.675,4.5219)--(5.6399,4.4909)--(5.60494,4.46)--
(5.57,4.429)--(5.53,4.398)--(5.4998,4.367)--(5.46,4.34)--(5.4297,4.3057)--(5.39468,4.2749)--
(5.3596,4.244)--(5.32,4.21)--(5.289,4.1826)--(5.2545,4.15)--(5.219,4.12)--(5.1844,4.090471141460904)--
(5.149,4.06)--(5.114,4.029)--(5.0792,3.9985)--(5.04,3.9678)--(5.0,3.937)--(4.97,3.9)--
(4.939,3.876)--(4.9,3.845)--(4.869,3.81)--(4.83,3.78)--(4.79,3.75)--(4.76,3.72)--
(4.7288,3.693)--(4.69,3.66)--(4.6587,3.63)--(4.62,3.6)--(4.588679,3.57)--(4.5536358173076925,3.54)--
(4.52,3.5)--(4.4835,3.48)--(4.4485,3.45)--(4.41,3.42)--(4.378,3.39)--(4.34,3.36)--
(4.31,3.33)--(4.27,3.2997)--(4.24,3.3)--(4.2,3.239)--(4.168,3.2)--(4.1,3.2)--
(4.098,3.15)--(4.063,3.119)--(4.03,3.089)--(3.99,3.0596)--(3.95,3.0297)--(3.9,2.9998)--
(3.88,2.97)--(3.85,2.94)--(3.8177,2.91)--(3.78,2.88)--(3.747,2.85)--(3.7,2.82)--
(3.6775,2.79)--(3.64,2.76)--(3.607,2.73)--(3.57,2.7028)--(3.537,2.67)--(3.502,2.64)--
(3.467,2.6143)--(3.43,2.58)--(3.397,2.55)--(3.362,2.53)--(3.34,2.497)--(3.29,2.467)--
(3.257,2.438)--(3.22198,2.40888)--(3.1869,2.379)--(3.15,2.35)--(3.1168,2.32)--(3.0818,2.29)--
(3.04676,2.263)--(3.012,2.234)--(2.976,2.205)--(2.941,2.176)--(2.90,2.147)--(2.87,2.12)--
(2.84,2.0896)--(2.8,2.06)--(2.766,2.03)--(2.731,2.0034)--(2.696,1.97)--(2.66,1.946)--
(2.626,1.91755)--(2.59,1.889)--(2.556,1.86)--(2.52,1.83)--(2.486,1.8037)--(2.451,1.775)--(2.42,1.747)--(2.38,1.7188)--(2.35,1.69)--(2.31,1.66)--(2.275,1.63)--(2.24,1.606)--
(2.2057,1.578)--(2.17,1.55)--(2.1356,1.52)--(2.10,1.49466)--(2.0655,1.46688)--(2.03,1.439)--
(1.995,1.4115)--(1.96,1.3839)--(1.925,1.356)--(1.89,1.328889)--(1.855,1.30)--(1.82,1.274)--
(1.785,1.2468)--(1.75,1.2196)--(1.715,1.19)--(1.68,1.165)--(1.645,1.138)--(1.61,1.1114)--
(1.5749,1.0845)--(1.54,1.0577)--(1.5048,1.031)--(1.4698,1.004)--(1.434758,0.9778)--(1.3997,0.95)--
(1.36467,0.9249)--(1.3296,0.8986)--(1.29,0.87)--(1.259,0.846)--(1.224,0.82)--(1.189,0.794)--
(1.154,0.768)--(1.119,0.74)--(1.084,0.7169)--(1.049,0.69)--(1.014,0.666)--(0.979,0.64)--
(0.944,0.615)--(0.92,0.59)--(0.87,0.565)--(0.839,0.54)--(0.8,0.52)--(0.7689,0.49)--
(0.73,0.466)--(0.6988,0.44)--(0.66,0.42)--(0.63,0.39)--(0.594,0.37)--(0.558,0.346)--
(0.52,0.32)--(0.488,0.299)--(0.45,0.276)--(0.418,0.253)--(0.38,0.23)--(0.35,0.21)--
(0.31,0.185)--(0.278,0.163)--(0.243,0.141)--(0.208,0.1197)--(0.173,0.098)--(0.138,0.077)--
(0.103,0.057)--(0.068,0.0369)--(0.033,0.0)--(0,0)--(6.0,0)--(6.3,0)--(6.3,4.95)--(6.0,4.86);
\draw[scale=1] (3.0,1.0) node[above,rotate=35] {\footnotesize Global stability};
\draw[scale=1] (5.3,4.2) node[below,rotate=39] { $r_\infty$};
\draw[scale=1] (0,1) node[left,rotate=0] {\scriptsize $1$};
\draw[scale=1] (3.8,3.7) node[above,rotate=45] {\footnotesize $r=h$};
\draw[scale=1] (1.5,3.5) node[above,rotate=30] {\footnotesize Instability};
\end{tikzpicture}
\end{center}
\caption{This figure shows the stability regions in the $(h,r)-$plane for several choices of the delay. The main curve is the bottom red curve representing $r=r_\infty,$ which shows the global stability region obtained by our theory.  The curves from top to bottom are as follows: The curve $r=r_0$  represents the boundary of the local and global stability regions when no delay is involved in the model, i.e., $k=0.$ The curve $r=r_1$ represents the boundary of the local stability region when the delay is $k=1.$ The curve $r=r_2$  represents the boundary of the local stability region when the delay is $k=2.$ The curve  $r=r_2$ can be found explicitly. The curve $r=r_3$  represents the boundary of the local stability region when the delay is $k=3.$ This curve is found numerically.  The curve  $r=r_\infty$ represents the boundary of the global stability region that is found based on our theory for any finite value of $k.$}\label{Fig-StabilityRegions1}
\end{figure}

\subsection{Rational difference equations}
Consider the rational difference equation
\begin{equation} \label{Eq-RationalExample}
x_{n+1}=F(x_n,\ldots,x_{n-k+1})=\frac{a_0+\sum_{j=0}^{k-1}a_{j+1}x_{n-j}}{b_0+\sum_{j=0}^{k-1}b_{j+1}x_{n-j}},
\end{equation}
where $a_0=b_0=1$ and, the initial conditions and the coefficients are all nonnegative real numbers. Here, we have $F:\;\mathbb{R}^k_+\to  \mathbb{R}_+.$ Eq. \eqref{Eq-RationalExample} covers a wide spectrum of rational difference equations \cite{Ca-La2008,Ko-La1993,Ku-La2002}.  Define
$$A:=\sum_{j=1}^{k}a_j\quad \text{and}\quad B:=\sum_{j=1}^{k}b_j.$$
To avoid the trivial case, we consider that not all coefficients are identically zero. Also, the linear case (i.e., $B=0$) is an interesting case that we consider elsewhere. Here, we assume $B\neq 0.$
Next, we investigate the equilibrium solutions of Eq. \eqref{Eq-RationalExample}. The equilibrium solutions are fixed points of the function
\begin{equation}\label{Eq-FixedPoints}
    y=L(x)=\frac{1+Ax}{1+Bx},
\end{equation}
which are the zeros of
\begin{equation}\label{Eq-Roots-p(x)}
    p(x)=1+(A-1)x-Bx^2.
\end{equation}
We have a unique positive equilibrium solution, as shown by
$$\bar y=\frac{A-1+\sqrt{(A-1)^2+4B}}{2B},\quad \text{where}\quad B\neq 0.$$
Now, we proceed to guarantee that $F$ is monotonic in each component. Define
 \begin{equation}\label{Eq-Dij}
 D_{i,j}:=det\left(\left[
                     \begin{array}{cc}
                       a_i & a_{j\bmod k} \\
                       b_i & b_{j\bmod k} \\
                     \end{array}
                   \right]\right).
 \end{equation}
The subsequent facts summarize possible scenarios that are needed in the sequel and are straightforward to verify.
 \begin{proposition}\label{Pr-MonotonicityOfF}
Consider Eq. \eqref{Eq-RationalExample}, and let $D_{i,j}$ be defined as in Eq. \eqref{Eq-Dij}. If $D_{i,j}\geq 0$ for each $j=1,2,\ldots,k,$ then $F$ is non-decreasing in its $ith$ component. Similarly, if
 $D_{i,j}< 0$ for each $j=1,2,\ldots,k,$ then $F$ is non-increasing in its $ith$ component.
\end{proposition}
Observe that $D_{i,j}=-D_{j,i},$ which forces a certain relationship between the coefficients. For instance, to have $F$ increasing in its $i$th and $i^*$th arguments, then the vectors $v_1=(a_i,b_i)$ and $v_2=(a_{i^*},b_{i^*})$ must be linearly dependent.  We give the following example to show some viable monotonicity options:

\begin{example}\label{Ex-RationalExamples}\rm
Consider Eq. \eqref{Eq-RationalExample} with delay 3 and $B\neq 0.$
\begin{description}
\item{(i)} $F(x,y,z)=\frac{1}{1 + x + y + z}$ is decreasing in all arguments
\item{(ii)} $F(x,y,z)=\frac{1 + 3x + 6y + z}{1 + 2x + 4y + 30z}$ satisfied $F(\uparrow,\uparrow,\downarrow)$
\item{(iii)} $F(x,y,z)=\frac{1 + 3x + 6y + 6z}{1 + 2x + 4y + 4z}$ is increasing in all arguments
\item{(iv)} $F(x,y,z)=\frac{1 + 3x}{1 + 2x + 4y + 2z}$ satisfies $F(\uparrow,\downarrow,\downarrow)$
\item{(v)}   $F(x,y,z)=\frac{1 + 3x + 3z}{1 + 2x + 4y + 2z}$ satisfies $F(\uparrow,\downarrow,\uparrow).$
\end{description}
\end{example}
It is clear now that all monotonicity options are possible, and based on this, we proceed to give our global stability result.
Define
\begin{equation}\label{Eq-Gamma0AndGamma1}
\Gamma_j=\{i:\; (-1)^jF\;\; \text{is increasing in its ith argument}\}.
\end{equation}
\begin{proposition}\label{Pr-RationalArtificial1}
Define the partial order:
$\leq_\tau$ as given in Definition \ref{Def-Poset} to be compatible with the monotonicity of $F.$ Let $G:\; \mathbb{R}_+^k\times \mathbb{R}_+^k\to  \mathbb{R}_+^k\times \mathbb{R}_+^k$ be the diagonal extension of $F.$ If $\Gamma_0$ or $\Gamma_1$ is empty, then the only fixed point of $G$ is $\bar \xi=(\bar y,\ldots,\bar y).$
\end{proposition}
\begin{proof}
If $D_{i,j}\geq 0$ for all $i$ and $j,$ then $G(\xi)=\xi$ leads to solving the equations
$$x=\frac{1+xA}{1+xB}\quad \text{and}\quad y=\frac{1+yA}{1+yB}.$$
The only solution is $x=y=\bar y.$
If $D_{i,j}\leq 0$ for all $i$ and $j,$ then $G(\xi)=\xi$ leads to solving the equations
$$\frac{1+Ax}{1+Bx}=y\quad \text{and}\quad x=\frac{1+Ay}{1+By}.$$
Again, here, the only solution is $x= y=\bar y.$
\end{proof}
\begin{proposition}\label{Pr-RationalArtificial2}
Suppose neither $\Gamma_0$ nor $\Gamma_1$ is empty. Let $A_i:=\sum_{i\in\Gamma_i}a_i,$  $  B_i:=\sum_{i\in\Gamma_i}b_i,$ $\hat{A}:=A_0-A_1-1,$ $\tilde{B}=B_0-B_1$, $\beta=\frac{\hat{A}}{2B_0}$ and $B^*:=\frac{B_0(4B_0+4A_1\hat{A}+\hat{A}^2)}{\hat{A}^2}.$
Define the partial order $<_\tau$ as given in Definition \ref{Def-Poset} to be compatible with the monotonicity of $F$. Let $G:\; \mathbb{R}_+^k\times \mathbb{R}_+^k\to  \mathbb{R}_+^k\times \mathbb{R}_+^k$ be the extension of $F.$  Each of the following holds true:
\begin{description}
\item{(I)} $G$ has a unique fixed point if one of the following is satisfied:
\begin{description}
\item{(i)} $B_1\leq B_0$
\item{(ii)} $B_1>B_0$ and $ \hat{A}\leq 0$
 \item{(iii)} $B_1>B_0,$ $ \hat{A}> 0$ and $B_1\leq B^*.$
\end{description}
\item{(II)}  If $B_1>B_0,$ $\hat{A}>0$ and $B_1>B^*,$ then $G$ has three fixed points  $(P_\tau,P_\tau^t),$ where $P_\tau$ as defined in Eq. \eqref{Eq-Ptau} for $(x,y)\in(\bar y,\bar y), (t_0,t_1),(t_1,t_0)$ and
$$t_i=\beta-(-1)^i\sqrt{\frac{1}{\tilde{B}}\left(\tilde{B}\beta^2+2\beta A_1+1\right)}.$$
\end{description}
\end{proposition}
\begin{proof}
Suppose that for all values of $j,$  $D_{i,j}\geq 0$ for some values of $i$ while $D_{i,j}\leq 0$ for other values of $i.$ In this case,
$G(\xi)=\xi$ leads to solving

$$x=\frac{1+A_0x+A_1y}{1+B_0x+B_1y}\quad \text{and}\quad y=\frac{1+A_0y+A_1x}{1+B_0y+B_1x}, $$
or equivalently
\begin{equation}\label{Eq-ArtificialFixed1}
(A_1-B_1x)y=B_0x^2+(1-A_0)x-1\quad \text{and}\quad  (A_1-B_1y)x=B_0y^2+(1-A_0)y-1.
\end{equation}
Obviously, the two equations represent hyperbolas unless $\bar y=\frac{A_1}{B_1}.$ In this particular scenario, the two equations can be simplified to  $x=y=\bar y.$ Also,  $x=y=\bar y$ is a solution regardless of the value of $\bar y.$  Therefore, we proceed assuming that $x\neq y$ and the curves of Eqs. \eqref{Eq-ArtificialFixed1} are hyperbolas. We begin by considering the case $B_1=B_0.$ Substitute and re-write Eqs. \eqref{Eq-ArtificialFixed1} as
$$\begin{cases}
y=&-x+\frac{\hat{A}}{B_0}+\frac{A_1\hat{A}+B_0}{B_0(B_0x-A_1)}\\
x=&-y+\frac{\hat{A}}{B_0}+\frac{A_1\hat{A}+B_0}{B_0(B_0y-A_1)}.
\end{cases}
$$
By subtracting the two equations, we observe that $x$ cannot be different from $y.$ Next, consider the case $B_1<B_0.$
In this case, $x$ and $y$ must be positive zeros of the quadratic polynomial
$$q(x)=B_0(B_1-B_0)x^2-(B_1-B_0)\hat{A}x+A_1\hat{A}+B_0.$$
Based on the coefficients' signs, we use Descart's rule of signs to conclude that it is impossible to have both zeros of $p$ positive.
Next, assume $B_1> B_0.$ If $\hat{A}\leq 0,$ then $q(x)-q(y)=0$ gives $x+y=\frac{\hat{A}}{B_0}\leq 0.$ So, the two zeros of $q$ cannot be positive. Thus, we proceed with the $\hat{A}>0$ case.  Based on the signs of the coefficients of $q$, we obtain two positive zeros or none. Indeed, two positive zeros bifurcate from $\bar y$ when $B_1=B^*.$ Therefore, we obtain Part (iii) of Case (I) when $B_1\leq B^*$    and Case (II) when $B>B^*.$ It is a computational matter to find the explicit form of the fixed points, and we omit it.
\end{proof}
Note that Part (ii) of Example \ref{Ex-RationalExamples} can be used to illustrate Case (II) of Proposition \ref{Pr-RationalArtificial2}. Indeed, we have $A_0=9,B_0=6,A_1=1$ and $B_1=30.$ This gives us $\bar y=\frac{1}{3},$ and the other two solutions are $\left(\frac{1}{12},\frac{13}{12}\right)$ and $\left(\frac{13}{12},{1}{12}\right).$ Before we give the global stability result, we find it convenient to re-write Eqs. \eqref{Eq-ArtificialFixed1} differently. Define
\begin{align}\label{Eq-q2}
q_2(t)=& \frac{B_0t^2+(1-A_0)t-1}{A_1-B_1t},\quad t\neq \frac{A_1}{B_1}\\ \nonumber
=& -\frac{B_0}{B_1}t+\frac{\Delta}{B_1^2}-1-\frac{p\left(\frac{A_1}{B_1}\right)}{A_1-B_1t},
\end{align}
where $\Delta=A_0B_1-A_1B_0$ and $p$ is the polynomial defined in Eq. \eqref{Eq-Roots-p(x)}. This makes the geometric representation of the hyperbolas in Eqs.   \eqref{Eq-ArtificialFixed1} possible based on whether the equilibrium is larger or smaller than $\frac{A_1}{B_1}.$ However, we show that $\bar y<\frac{A_1}{B_1}$ is an invalid option.

\begin{lemma}\label{Pr-SimpleFacts}
Let $\Gamma_0$ and $\Gamma_1$ be nonempty, and consider $\Delta=B_1A_0-A_1B_0$. Suppose $B_0B_1\neq 0$. Each of the following holds true:
\begin{description}
\item{(i)} $\frac{A_0}{B_0}\geq \frac{A}{B}\geq \frac{A_1}{B_1}$
\item{(ii)} $\Delta\geq 0$
\item{(iii)} $A_0\geq B_0$ and $A_1\leq B_1$
\item{(iv)} $\bar y> \frac{A_1}{B_1}.$
\end{description}
\end{lemma}
\begin{proof}
For $i\in \Gamma_0,$ we have $a_1b_j\geq b_ia_j$ for all $j\neq i.$ This gives us
$$a_i\sum_{j\neq i}b_j\geq b_i\sum_{j\neq i}a_j\quad \Leftrightarrow \quad a_iB\geq b_iA.$$
Sum over all $i\in \Gamma_0$ to obtain $A_0B\geq B_0A.$ Similarly, $A_1B\leq B_1A.$ This clarifies Part (i) and Part (ii). Part (iii) is obvious. Finally, to prove Part (iv), observe that Part (iv) is valid if and only if $p\left(\frac{A_1}{B_1}\right)\geq 0.$ Since $A_1\leq B_1$ and $\Delta\geq 0,$ we obtain
$$B_1(A_1-B_1)\leq A_1\Delta,$$
which is equivalent to
$$p\left(\frac{A_1}{B_1}\right)=\frac{A_1}{B_1^2}(\Delta-B_1)+1\geq 0.$$
Next, assume $\bar y= \frac{A_1}{B_1}.$ Since $p$ is decreasing on the positive real numbers, we obtain
$$\frac{A_1}{B_1}\leq \frac{A}{B}\leq \frac{A_0}{B_0}\;\;\Rightarrow\;
p\left(\frac{A_1}{B_1}\right)\geq p\left(\frac{A}{B}\right)\geq p\left( \frac{A_0}{B_0}\right).$$
Because $p\left(\frac{A_1}{B_1}\right)=0$ and $p\left(\frac{A}{B}\right)=1-\frac{A}{B}$, we obtain $A\geq B.$ But the curve of Eq. \eqref{Eq-FixedPoints} has $1$ as $y$-intercept and $y=\frac{A}{B}$ as the horizontal asymptote. The action converts the function into a constant value, which we intentionally avoided during the preliminary phase.
\end{proof}
\begin{theorem}
Consider Eq. \eqref{Eq-RationalExample}, and let $\Gamma_i$ as defined in Eq. \eqref{Eq-Gamma0AndGamma1}. Also, consider $A_i,B_i,\hat{A}$ and $B^*$ as defined in Proposition \ref{Pr-RationalArtificial2}. Each of the following holds true:
\begin{description}
\item{(I)} If $\Gamma_0$ or $\Gamma_1$ is empty, then $\bar y$ is a global attractor.
\item{(II)} Suppose neither $\Gamma_0$ nor $\Gamma_1$ is empty. Then $\bar y$ is a global attractor if one of the following conditions is satisfied:
\begin{description}
\item{(i)} $B_1\leq B_0$
\item{(ii)} $B_1>B_0$ and $ \hat{A}\leq 0$
 \item{(iii)} $B_1>B_0,$ $ \hat{A}> 0$ and $B_1\leq B^*.$
\end{description}
\end{description}
\end{theorem}
\begin{proof}
(I) Suppose $\Gamma_1$ is empty, or equivalently, $D_{i,j}\geq 0$ for all $i$ and $j.$ This means $F$ is non-decreasing in each one of its components. Let $a<b$ and $\xi=(X,Y),$ where $X$ and $Y$ are composed of $k$ arguments with $a$ in each argument of $X$ and $b$ in each argument of $Y.$ Define the order $\leq_{\tau}$ to be compatible with the monotonicity of $F$ as in Definition \ref{Def-Poset}  and
$G(\xi)=(F(X),a,\ldots,a,F(Y),b,\ldots,b).$ From Proposition \ref{Pr-RationalArtificial2}, $G$ has a unique fixed point. Now, $\xi<_{\lambda}G(\xi)$ gives us
$$a<\frac{1+aA}{1+aB}\quad \text{and}\quad b>\frac{1+bA}{1+bB}.$$
The feasible set of solutions for these inequalities is $\{(a,b):\; 0<a<\bar y, b>\bar y\}$. Therefore, for each initial condition $X_0=(x_0,x_{-1},\ldots,x_{-k+1}),$ there exists $a<b$ such that $\xi<_{\lambda} (X_0,X_0)<_{\lambda}\xi^t.$ By Theorem \ref{Th-GlobalStability}, $G^n(\xi)$ and $G^n(\xi^t)$ converge to $\bar \xi=(\bar y,\ldots,\bar y)$. Therefore, the orbits of Eq. \eqref{Eq-RationalExample} converge to $\bar y.$
 Suppose $\Gamma_1$ is empty, or equivalently, Suppose $D_{i,j}\leq 0$ for all $i$ and $j.$ This means $F$ is non-increasing in each component. As before, let $a<b$ and $\xi=(X,Y),$ where $X$ and $Y$ are composed of $k$ arguments with $b$ in each argument of $X$ and $a$ in each argument of $Y.$ Define the order $\leq_{\tau}$ to be compatible with the monotonicity of $F,$  and define
$G(\xi)=(F(Y),b,\ldots,b,F(X),a,\ldots,a).$ Again here, Proposition \ref{Pr-RationalArtificial2} shows that $G$ has a unique fixed. Now, $\xi\leq_{\lambda}G(\xi)$ gives us
$$\frac{1}{1+aB}\leq b\quad \text{and}\quad a\leq \frac{1}{1+bB}.$$
The feasible set of solutions for these inequalities is unbounded, and for each initial condition $X_0=(x_0,x_{-1},\ldots,x_{-k+1}),$ there exists $a<b$ such that $\xi<_{\lambda} (X_0,X_0)<_{\lambda}\xi^t.$
By Theorem \ref{Th-GlobalStability}, $G^n(\xi)$ and  $G^n(\xi^t)$ converge to $\bar \xi=(\bar y,\ldots,\bar y)$. Therefore, the orbits of Eq. \eqref{Eq-RationalExample} converge to $\bar y.$

(II) Suppose neither $\Gamma_0$ nor $\Gamma_1$ is empty. Define the order $\leq_{\tau}$ to be compatible with the monotonicity of $F.$ Let $x<y$ and define $X=(x_1,\ldots,x_k),$ where $x_i=x$ if $F$ is increasing in its $ith$ component while $x_i=y$ if $F$ is decreasing in the $ith$ component. Let $Y=X^t,$ i.e., switch $x\leftrightarrow y$ in $X$ to obtain $Y.$  Let $\xi=(X,Y)$ and define $G$ as in Definition \ref{Def-Gtau}. We obtain $\xi\leq_\tau G(\xi)$ if
$x\leq F(X) $ and $F(Y)\leq y.$ This implies
$$x\leq \frac{1+A_0x+A_1y}{1+B_0x+B_1y}\quad \text{and}\quad y\geq \frac{1+A_0y+A_1x}{1+B_0y+B_1x}. $$
Therefore, we need to solve the system of inequalities $x<y$ and
\begin{equation}\label{In-ArtificialFixed1}
\begin{cases}
 (A_1-B_1x)y&>\;B_0x^2+(1-A_0)x-1\\
 (A_1-B_1y)x&<\;B_0y^2+(1-A_0)y-1.
\end{cases}
\end{equation}
Based on Lemma \ref{Pr-SimpleFacts}, we need to focus on $\bar y >\frac{A_1}{B_1}.$  The feasible region of the inequalities is determined by the boundary curves, which are two hyperbolas. Furthermore, Conditions  (i), (ii) or (iii) ensure the intersection takes place at the unique equilibrium point. Therefore, the asymptotes of the two hyperbolas are sufficient to determine an unbounded feasible region. Based on the function $q_2$ of Eq. \eqref{Eq-q2}, the asymptotes are
$$x=\frac{A_1}{B_1},\quad y=-\frac{B_0}{B_1}x+\frac{\Delta}{B_1^2}-1$$
and
$$y=\frac{A_1}{B_1},\quad x=-\frac{B_0}{B_1}y+\frac{\Delta}{B_1^2}-1.$$
Figure \ref{Fig-Rational} illustrates the feasible region in this case.
Therefore, for any initial condition $(x_0,\ldots,x_{1-k}),$ we can find a point $(a,b)$ in the feasible region such that $x_0,\ldots,x_{1-k}\in [a,b].$ The rest of the proof is the same as in Part (I).
\end{proof}
\definecolor{ffqqqq}{rgb}{1.,0.,0.}
\definecolor{qqqqff}{rgb}{0,0,1}
\definecolor{zzttqq}{rgb}{0.6,0.2,0}
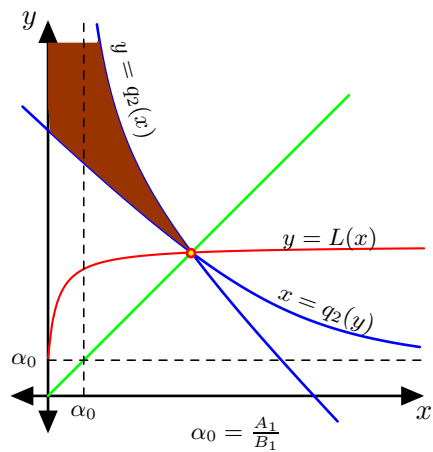
\begin{figure}[h!]
\centering
\begin{minipage}[t]{0.8\textwidth}
\raggedright
\begin{center}
\begin{tikzpicture}[line cap=round,line join=round,>=triangle 45,x=1.0cm,y=1.0cm,scale=0.5]
\draw[-triangle 45, line width=1.0pt,scale=1] (0,0) -- (10.0,0) node[below] {$x$};
\draw[line width=1.0pt,-triangle 45] (0,0) -- (-1,0);
\draw[-triangle 45, line width=1.0pt,scale=1] (0,0) -- (0,10) node[left] {$y$};
\draw[line width=1.0pt,-triangle 45] (0,0) -- (0.0,-1);
\draw[line width=1.0pt,domain=0:8.0,smooth,variable=\x,green] plot ({\x}, {\x)});
\draw[line width=0.8pt,color=red,smooth,samples=100,domain=0.0:10.0] plot(\x,{(1+16*\x)/(1+4*\x)});
\draw[line width=1pt,color=blue,smooth,samples=100,domain=1.3:7.7] plot(\x,{(2*\x*\x-14*\x-1)/(1-2*\x)});
\draw[line width=1pt,color=blue,smooth,samples=100,domain=1.3:7.7] plot({(2*\x*\x-14*\x-1)/(1-2*\x)},\x);
\fill[line width=2pt,color=zzttqq,fill=zzttqq,fill opacity=0.300]  (3.81,3.82)--(3.55,4.02)--(3.29,4.22)--(3.04,4.42)--(2.79,4.62)--
(2.55,4.82)--(2.31,5.02)--(2.07,5.22)--(1.84,5.42)--(1.61,5.62)--(1.38,	5.82)--(1.16,6.02)--(0.94,6.22)--(0.71,6.42)--(0.49,6.62)--(0.27,6.82)--
(0.06,7.02)--(0.0,7.707)--(0.0,9.4)--(1.3,9.4)--(1.35,9.5)--(1.55,8.52)--(1.75,7.75)--(1.95,7.14)--(2.15,6.62)--(2.35,6.18)--(
2.55,5.78)--(2.75,5.42)--(2.95,5.08)--(3.15,4.77)--(
3.35,4.47)--(3.55,4.18)--(3.75,3.90) -- cycle;
\draw[scale=1] (0,0.95) node[left] {\footnotesize $\alpha_0 $};
\draw[scale=1] (0.95,0) node[below] {\footnotesize $\alpha_0 $};
\draw[scale=1] (7.5,3.6) node[above] {\footnotesize $y=L(x) $};
\draw[scale=1] (7.2,1.75) node[above,rotate=-20] {\footnotesize $x=q_2(y) $};
\draw[scale=1] (1.7,8.0) node[above,rotate=-75] {\footnotesize $y=q_2(x) $};
\draw[line width=0.6pt,dashed,color=black] (0.95,0) -- (0.95,10);
\draw[line width=0.6pt,dashed,color=black] (0,0.95) -- (10.0,0.95);
\draw[line width=1.0pt,red,fill=yellow] (3.8,3.8) circle (3.0pt);
\draw[scale=1] (5,-0.3) node[below] {\footnotesize  $\alpha_0=\frac{A_1}{B_1}$};
\end{tikzpicture}
\end{center}
\end{minipage}%
\caption{This figure shows the feasible region of the inequalities in \eqref{In-ArtificialFixed1} when $x<y.$ The functions $L$ and $q_2$ are given in Eq. \eqref{Eq-FixedPoints} and Eq. \eqref{Eq-q2}, respectively.    }\label{Fig-Rational}
\end{figure}

\section{Conclusion}
In this paper, we considered $k$-dimensional maps of mixed monotonicity and introduced compatible partial orders. Then, we embedded the $k$-dimensional system into monotonic $2k$-dimensional system. The orbits of the embedded system are used to squeeze the orbits of the original system. Subsequently, bounded monotonic sequences and their convergence are utilized to establish global stability. This conclusion was generalized to periodic difference equations of the form
$$x_{n+1}=F_n(x_n,x_{n-1},\ldots,x_{n-k+1}).$$
The novelty of our theory lies in its generality to tackle $k$-dimensional systems for any $k$.
As an application of our developed theory, we discussed the global stability of the $k$-dimensional Ricker model
$$x_{n+1} =x_ne^{r-x_{n-k}}+h,\; r,h>0$$
and the rational difference equation
$$
x_{n+1}=F(x_n,\ldots,x_{n-k+1})=\frac{1+\sum_{j=0}^{k-1}a_{j+1}x_{n-j}}{1+\sum_{j=0}^{k-1}b_{j+1}x_{n-j}},$$
where the initial conditions and the coefficients are all nonnegative real numbers. Additionally, constraints have been imposed on the coefficients to ensure monotonicity in each argument. In the given case of the Ricker model, sufficient conditions that ensure global stability were derived. Finding necessary and sufficient conditions for global stability for the general case is an open problem. The obtained sufficient conditions surprisingly do not rely on the delay value $k$. This finding is of paramount significance since it challenges the prevailing notion that delay has a destabilizing effect.  Indeed, the observed effect on local stability seems to diminish as the magnitude of the delay increases.
  In the example of a rational equation, we established the existence of a global attractor by imposing certain constraints on the present coefficients. This outcome encompasses a broad range of rational difference equations that appear in the literature.
\bigskip

\noindent{\textsc{Acknowledgement:}} The first author is supported by sabbatical leave from the American University of Sharjah and a Maria Zambrano grant for attracting international talent from the Polytechnic University of Cartagena.
\bibliographystyle{unsrt}

\bibliographystyle{plain[8pt]}

\end{document}